\documentclass[11pt]{article}
\usepackage [dvips]{graphics}
\usepackage{amsmath}
\usepackage{amsfonts}
\usepackage{amssymb}
\usepackage{amsthm}
\usepackage{newlfont}
\usepackage{stmaryrd}
\usepackage{color,epsfig}
\usepackage{amsfonts,graphicx,psfrag}
\usepackage{graphicx}
\usepackage{float}
\usepackage{subcaption}

\usepackage{txfonts}
\usepackage{mathrsfs}
\usepackage{verbatim}
\usepackage{cite}

\theoremstyle{plain}
\newtheorem{thm}{Theorem}[section]
\newtheorem{cor}[thm]{Corollary}
\newtheorem{lem}[thm]{Lemma}
\newtheorem{prop}[thm]{Proposition}

\newtheorem{rem}[thm]{Remark}
\newtheorem{defi}[thm]{Definition}

\def\sqr#1#2{{\vcenter{\vbox{\hrule height.#2pt
              \hbox{\vrule width.#2pt height#1pt \kern#1pt \vrule
width.#2pt}
              \hrule height.#2pt}}}}
\def\be{\begin{equation}}
\def\ee{\end{equation}}

\def\ga{{\gamma}}

\def\ep{{\epsilon}}

\def\Sp{{\mathrm {Sp}}}
\def\lb{\label}

\def\ga{{\gamma}}

\def\R{{\Bbb R}}

\def\<{{\langle}}
\def\>{{\rangle}}
\def\no{\noindent}

\def\bs{\bigskip}

\def\dim{\hbox{\rm dim$\,$}}

\def\({\Big (}
\def\){\Big )}
\def\[{\Big[}
\def\]{\Big]}

\def\be{\begin{equation}}
\def\bel{\begin{equation}\label}
\def\ee{\end{equation}}
\def\bea{\begin{eqnarray}}
\def\eea{\end{eqnarray}}
\def\bt{\begin{theorem}}
\def\et{\end{theorem}}
\def\bc{\begin{corollary}}
\def\ec{\end{corollary}}
\def\bl{\begin{lemma}}
\def\el{\end{lemma}}
\def\bp{\begin{proposition}}
\def\ep{\end{proposition}}
\def\br{\begin{remark}}
\def\er{\end{remark}}
\def\ba{\begin{array}}
\def\ea{\end{array}}
\def\bd{\begin{definition}}
\def\ed{\end{definition}}

\makeatletter
   
   \@addtoreset{equation}{section}
\makeatother

\begin{document}

\title{\bf Quantitative Linear Stability Analysis of Elliptic Relative Equilibria in the Planar $N$-Body Problem}
\author{Xijun Hu\thanks{E-mail:xjhu@sdu.edu.cn}
\quad Yuwei Ou\thanks{E-mail:ywou@sdu.edu.cn}
\quad Jiexin Sun \thanks{E-mail:202120273@mail.sdu.edu.cn.}
\\ \\
School of Mathematics, Shandong University
Jinan, Shandong 250100\\
The People's Republic of China
}
\date{}
\maketitle
\begin{abstract}
An elliptic relative equilibrium (ERE) is a special solution of the planar $N$-body problem generated by a central configuration. Its linear stability depends on the eccentricity $e$ and the masses of the bodies. However, for $e>0$, the variational equations become non-autonomous and highly complex, particularly near $e=1$, where the system exhibits a singularity. This complicates the stability analysis as $e$ approaches one, making it challenging to derive a rigorous quantitative estimate for the stable region across $e\in[0,1)$.
In this work, we address this problem. Using trace formulas for the non-degenerate Hamiltonian system of EREs, we establish an upper bound ensuring non-degeneracy for all $e\in[0,1)$. As key applications, we provide explicit stability estimates for the Lagrange, Euler, and regular $(1+n)$-gon EREs over the full range of eccentricity.
\end{abstract}

\bs

\no{\bf AMS Subject Classification:} 37J25,  70F10,   37J45, 53D12

\bs

\no{\bf Key Words: $N$-body problem, linear stability, elliptic relative equilibrium, Maslov-type index, Trace formula}

\tableofcontents


\subsection*{Notation}
\small{
The following notations will be used without further comments throughout the paper.
\begin{itemize}
\item  We denote the real number set, complex number set, the non-negative integer
set and the unit circle by $\mathbb{R}$, $\mathbb{C}$, $\mathbb{N}$ and $\mathbb{U}$ respectively.

\item This derivative symbol $'$ means $\frac{d}{dt}$ and $\cdot$ means $\frac{d}{d\theta}$ throughout the paper.
\item Let $I_j$ be the identity matrix on $\mathbb{R}^j$ and
$J_{2j}=\left( \begin{array}{cccc}0_j& -I_j \\
                                  I_j& 0_j \end{array}\right)$, $\mathbb{J}_n=diag(J_2,...,J_2)_{2n\times2n}$. For simplicity, sometimes we omit the sub-indices of $I$ and $\mathbb{J}$,
but can be easily found out through the context.
\item Given a function $f: \mathbb{R}^{k} \to \mathbb{R}$ and matrix $C$, $\nabla f$ represents the gradient of $f$ with respect to the Euclidean inner product expressed as a column vector and
$D^2 f$ denote the Hessian of $f$, $C^{\mathcal{T}}$ denote the transpose matrix of $C$.
\item We denote by $GL(\R^{2n})$ the invertible matrix group, by $\mathcal{S}(2n)$ the symmetric matrix group, by $O(2n)$ the orthogonal matrix group in $\mathbb{R}^{2n}$, and by
$$   \Sp(2n)=\{M\in GL(\R^{2n}),  M^TJM=J\}  $$
the symplectic group.
\item As in \cite{L02}, for $M_1=\left( \begin{array}{cccc}A_1& A_2 \\
A_3 & A_4 \end{array}\right)$,  $M_2=\left( \begin{array}{cccc}B_1& B_2 \\
B_3 & B_4 \end{array}\right)$, the symplectic sum $\diamond$ is defined by
\bea  M_1\diamond M_2=\left( \begin{array}{cccccccc}A_1& 0 &A_2 & 0 \\
                                                     0 & B_1& 0  & B_2 \\
                                                    A_3& 0 &A_4 & 0 \\
                                                     0 & B_3& 0  & B_4 \end{array}\right).\nonumber   \eea
\item In what follows we write $A\geq B$ for two linear symmetric
operators $A$ and $B$, if $A-B\geq 0$, i.e. $A-B$ possesses only non-negative eigenvalues and
write $A>B$, if $A-B>0$, i.e. $A-B$ possesses only positive eigenvalues.
\end{itemize}
}
\section{Introduction and main results}
Consider $N$ particles with masses $m_1,\cdots,m_N$. Let $q=(q_1,\cdots,q_N)\in \mathbb{R}^{2N}$ be the position vector in the configuration space and $p=(p_1,\cdots,p_N)\in \mathbb{R}^{2N}$ be the momentum vector. We will consider the situation with configuration space
$$
\Lambda=\{x=(x_1,\cdots,x_N)\in\mathbb{R}^{2N}\setminus\triangle:
\sum_{i=1}^Nm_ix_i=0 \},$$
where $\triangle=\{x\in\mathbb{R}^{2N}:\exists\, i\neq j,x_i=x_j \}$ is the collision set.

Let
\bea U(q)=\sum_{1\leq i< j\leq N} \frac{m_im_j}{\|q_i-q_j\|} \nonumber \eea
be the negative potential function defined on $\Lambda$, and the Newton's equations are
\bea m_iq''_i(t)=\frac{\partial
U}{\partial q_i}(q_1,...,q_N),\quad i=1,\cdots,N. \label{eq1.1} \eea
The corresponding Hamiltonian system of (\ref{eq1.1}) has the form
\bea\label{eq H}
x'(t)=J_{4N}\nabla H(x)
\eea
with $x=(p,q)^{\mathcal{T}}\in \mathbb{R}^{4N}$, $J_{4N}=\left(\begin{array}{cc}
    O_{2N} & -I_{2N} \\
    I_{2N} & O_{2N}
\end{array}\right)$ and Hamiltonian functional
\bea
H(p,q)=\sum_{j=1}^N\frac{||p_{j}||^2}{2m_{j}}-U(q),\nonumber
\eea
where $I_{2N}$ is the $2N\times 2N$ identical matrix.
For a periodic solution $x(t)$ of (\ref{eq H}), the corresponding fundamental solution matrix $\xi$ satisfies the linearized Hamiltonian system at $x$
\bea\label{ga}
\xi'(t)=J_{4N}D^{2}H(x(t))\xi(t),\,\ \xi(0)=I_{4N},
\eea

A planar central configuration of $n$ particles with center of mass at original point is formed by a  $n$-position vector $a=(a_1,...,a_n)\in \mathbb{R}^{2n}$ which satisfies
\bea -\lambda \mathcal{M}a=\nabla U(a), \label{CC}\eea
for constant $\lambda=U(a)/\cal{I}(a)>0$, where $\mathcal{M}=diag(m_{1},m_{1},m_{2},m_{2}\ldots,m_{n},m_{n})$,
$\cal{I}(a)=\sum m_j\|a_j\|^2$ is the moment of inertia. Let $\Sigma=\{x\in \Lambda : \cal{I}(x)=1\}$, that $a=(a_{1},\cdots,a_{n})$ satisfies (\ref{CC}) implies that $a$ is a critical point of $U|_{\Sigma}$. It is well known that a planar central configuration of the $n$-body problem gives rise to a solution of
(\ref{eq H}) where each particle moves on a specific Keplerian orbit while the totality of the particles move on a homographic motion. More precisely,
the homographic solution generated by the central configuration $a$ is
\bea\label{ERE}
x(t)=r(t)\mathfrak{R}(\theta(t))a,\nonumber
\eea
where
$$
r(t)=\frac{\Omega^2/\lambda}{1+e\cos\theta(t)}, \ \ r^{2}(t)\theta'(t)=\Omega,
$$
and $\mathfrak{R}(\theta(t))=diag(R(\theta(t)),\ldots, R(\theta(t)))$, $R(\theta(t))=\left( \begin{array}{cc} \cos \theta(t) & -\sin \theta(t)\\
\sin \theta(t) & \cos \theta(t)\end{array}\right)$, $\Omega\neq0$ is the angular momentum.
If the Keplerian orbit
is elliptic then the solution is an equilibrium in pulsating coordinates so we call this solution an elliptic
relative equilibrium (ERE for short), and a relative equilibrium (RE for short) in case $e=0$ (cf. \cite{MS05}).
%

In order to study the linear stability of the ERE, Meyer and Schmidt in \cite{MS05} introduced a useful linear
transformation which reduces system (\ref{ga}) into two parts symplectically, one is corresponding to the Keplerian motion, and the other called essential part $B(\theta)$ is
closely related to the linear stability of the ERE. Let $\ga(\theta)$ be the fundamental solution of essential part, that is
\bea \dot{\ga}(\theta)=JB(\theta)\ga(\theta),  \quad  \ga(0)=I_{2n}. \label{ga3}
\eea
The monodromy matrix $\ga(2\pi)$ is called spectrally stable if all eigenvalues of $\ga(2\pi)$ belong to the
unit circle $\mathbb{U}$ in the complex plane $\mathbb{C}$. $\ga(2\pi)$ is called linearly stable if it is spectrally stable and semi-simple.
While $\ga(2\pi)$ is called hyperbolic if no eigenvalue of $\ga(2\pi)$ is on $\mathbb{U}$. The ERE is called spectrally stable
(linearly stable, hyperbolic, resp.) if the monodromy matrix $\ga(2\pi)$ is spectrally stable(linearly stable, hyperbolic, resp.).

There are many famous and interesting EREs in planar $N$-body problem which have been studied for a long time. We will focus on their linear stability. Lagrange solution in planar 3-body problem is found by Lagrange in 1772, which forms a equilateral triangle all the time (See Figure \ref{fig:Lagrange}).
\begin{figure}[H]
    \centering
    \includegraphics[width=5.4cm]{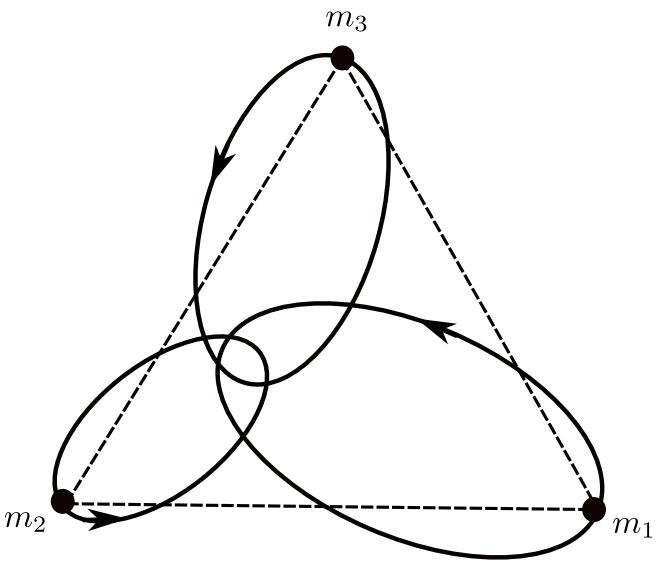}
    \caption{The Lagrange solution in planar 3-body problem.}
    \label{fig:Lagrange}
\end{figure}
The study on the stability of Lagrange solution has a long history. From Gascheau \cite{G43} in 1843 for circle Lagrange solution to Danby\ \cite{D64} in 1964 for elliptic case, the stability of Lagrange solution can be described by two parameters, mass parameter
\bea
\beta_{L}=\frac{27(m_1m_2+m_1m_3+m_2m_3)}{(m_1+m_2+m_3)^2} \in [0,9]\nonumber
\eea
and eccentricity $e\in [0,1)$.

Euler solution is another type long-historical ERE in planar 3-body problem discovered by Euler in 1767, which keeps collinear all the time (See Figure \ref{fig:Euler} ).
\begin{figure}[H]
    \centering
    \includegraphics[width=6.8cm]{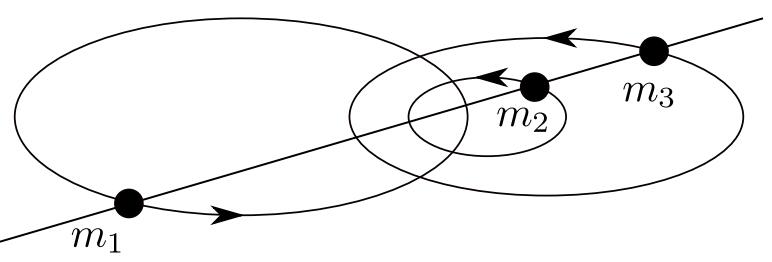}
    \caption{The Euler solution in planar 3-body problem.}
    \label{fig:Euler}
\end{figure}
The stability of Euler solution can be described by
\bea
\beta_{E}=\frac{m_1(3\mu^2+3\mu+1)+m_3\mu^2(\mu^2+3\mu+3)}{\mu^2+m_2[(\mu+1)^2(\mu^2+1)-\mu^2]} \in [0,7]\nonumber
\eea
and eccentricity $e\in [0,1)$, where $\mu$ is the unique positive solution of the Euler quintic polynomial equation, decided by the Euler configuration. The details of Euler solution can be found in \cite{ZL17}.

For $N>3$, there is a kind of ERE with nice symmetry called the regular $(1+n)$-gon central configuration solution, which has $n$ unit masses at the
vertices of a regular $n$-gon and a body of mass $m$ at the center (See Figure \ref{fig:1+n-gon} ). The stability of the regular $(1+n)$-gon can be described directly by
\bea
\beta_{M}=\frac{1}{m}\in (0,\infty),\nonumber
\eea
and eccentricity $e\in [0,1)$. The research of the linear stability of the regular $(1+n)$-gon ERE with $e=0$ was first started by J.C.Maxwell in his study on the stability of
Saturn's rings (cf. \cite{M83}).
\begin{figure}[H]
    \centering
    \includegraphics[width=5.4cm]{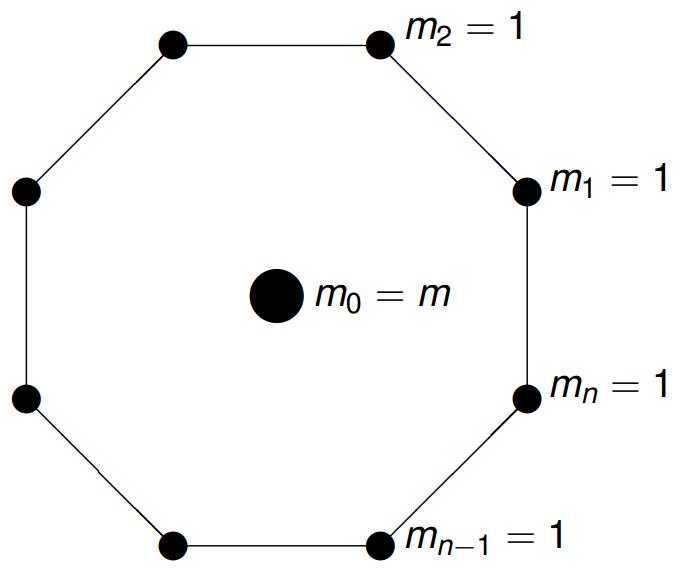}
    \caption{The regular $(1+n)$-gon solution in planar $N$-body problem.}
    \label{fig:1+n-gon}
\end{figure}

For Lagrange solution and Euler solution, the $4$-dimensional essential part  $B(\theta)=B_{\beta,e}(\theta)$ depends on the mass parameter $\beta=\beta_{L}, \beta_{E}$ and the eccentricity $e\in[0,1)$. As in \cite{MS05} the parameter plane $(\beta,e)$ can be partitioned into several parameter regions with the following notations:

\vspace{5pt}
\ $\bullet$\ \  Elliptic-Elliptic (EE): If $\gamma_{\beta,e}(2\pi)$ possesses eigenvalues only in $\mathbb{U}\setminus\mathbb{R}$;

\vspace{3pt}
\ $\bullet$\ \  Elliptic-Hyperbolic (EH): If $\gamma_{\beta,e}(2\pi)$ possesses eigenvalues both in $\mathbb{U}\setminus\mathbb{R}$ and $\mathbb{R}\setminus \{0,\pm 1\}$;

\vspace{3pt}
\ $\bullet$\ \  Hyperbolic-Hyperbolic (HH): If $\gamma_{\beta,e}(2\pi)$ possesses eigenvalues only in $\mathbb{R}\setminus\{0,\pm 1\}$;

\vspace{3pt}
\ $\bullet$\ \  Complex-Saddle (CS): If $\gamma_{\beta,e}(2\pi)$ possesses eigenvalues only in $\mathbb{C}\setminus(\mathbb{U}\cup \mathbb{R})$.

\vspace{5pt}
In \cite{MSS06}, Mart\'{\i}nez, Sam\`{a} and Sim\'{o} obtained the complete bifurcation diagrams numerically to describe the parameter regions of Lagrange solution and Euler solution and beautiful figures were drawn there for the full $(\beta,e)$ range. Hu and Sun firstly introduced Maslov-type index into the study of stability for Lagrange orbit in \cite{HS10}. Later in \cite{HLS14}, Hu, Long and Sun further developed this method and gave a complete analytically description of the bifurcation diagrams of \cite{MSS06} for Lagrange orbit. Zhou and Long in \cite{ZL17} study the stability of Euler solution with the similar method. Hu and Ou in \cite{HO16} develop the collision index to study the bifurcation diagram of Euler solution for the limit case (i.e $e\rightarrow1$).

For the regular $(1+n)$-gon ERE, $B(\theta)=B_{\beta,e}(\theta)$ is $(4n-4)$-dimensional and it can be decomposed into following form,
$$
B(\theta)=\hat{B}_1(\theta) \diamond\cdots\diamond \hat{B}_{[\frac{n}{2}]}(\theta).
$$
The expression of $\hat{B}_i(\theta), 1\leq i\leq{[\frac{n}{2}]}$ is given in (\ref{2.13}), it depends on mass parameter $\beta=\beta_{M}=1/m$ and eccentricity $e\in[0,1)$.
In \cite{M83}. Maxwell first proved that for $n\geq3$, it is linearly stable for sufficiently large $m$. But, Moeckel \cite{M92} found a mistake
in the calculation of Maxwell, then he corrected Maxwell's results and showed that the regular $(1+n)$-gon is linearly stable for sufficiently
large $m$ only when $n\geq7$. For $3\leq n\leq 6$, no matter how large the center mass $m$ is, it's not stable. Further, for $n\geq 7$, Roberts found a value $h_n$ which is proportional to $n^3$, and the regular $(1+n)$-gon is stable if and only if $m>h_n$ (cf. \cite{R98}). For other related works, please refer to \cite{VK07} and reference therein.
For the case $e>0$, the linear stability of the regular $(1+n)$-gon was first studied by Hu, Long and Ou in \cite{HLO20}, they showed that for $n\geq 8$
and any eccentricity $e\in[0,1)$, the regular $(1+n)$-gon ERE is linearly stable when the central mass $m$ is large enough. Later in \cite{OS22}, Ou and Sun also proved the regular $(1+7)$-gon is linearly stable when the central mass $m$
is large enough.

All the stability results above are the qualitative analysis. How to give a quantitative estimation for the stable region
of ERE through the rigorous analytical methods is a challenging problem. A useful method to estimate the parameter region of ERE is the trace formula. Hu, Ou and Wang in \cite{HOW15} and \cite{HOW19} developed the trace formula for Hamiltonian system and firstly used it to study the stability of elliptic relative equilibria quantitatively. They estimated the EE region for elliptic Lagrange solution and EH region for Euler solution with eccentricity $e$ less than some $e_0<0.34$.
But how to obtain the stable region over the full range $e\in[0,1)$ is still an open and difficult problem. Since near $e=1$, where the system exhibits a singularity, this complicates the stability analysis as $e$ approaches one, making it challenging to derive a rigorous quantitative estimate for the stable region
across $e\in[0,1)$.

In the present paper, we address this problem.
Our first observation is that when $\beta_{L}=\beta_{E}=0$, the essential part (\ref{ga3}) of the Lagrange solution and Euler solution can be reduced into
\bea
B(\theta)=B_{Kep}(e, \theta)=\left( \begin{array}{cccc} I_{2} & -J_{2} \\
J_{2} & I_{2}-\frac{\mathcal{R}_{kep}}{1+e\cos\theta}
\end{array}\right),\ \ \mathcal{R}_{kep}=\left( \begin{array}{cccc} 3 & 0 \\
0 & 0
\end{array}\right),\ \ \theta\in[0,2\pi]. \label{Kep}
\eea
in central configuration coordinate. Also, for the regular $(1+n)$-gon ERE, when
$\beta_{M}=0$ (\,i.e $m=+\infty$), $\hat{B}_i(\theta), 1\leq i\leq [\frac{n}{2}]$ which is given in (\ref{2.13}) can be decomposed into the form (\ref{Kep}).
In Section \ref{Sec2}, we will see that (\ref{Kep}) is precisely the representation of the linearized Kepler system in central configuration coordinate.
Therefore, the Lagrange solution, Euler solution and the regular (1 + n)-gon solution can be regarded as a perturbation of the linearized Kepler system.
Based on this observation, it's important to study the linear stability of the following general $2\mathfrak{n}$-dimensional linear Hamiltonian system
\bea\label{1.9e}
\dot{z}(\theta)=J_{2\mathfrak{n}}(B(\theta)+\sigma D(\theta))z(\theta), \quad z(0)=Sz(T),
\eea
where $S\in \Sp(2\mathfrak{n})\cap O(2\mathfrak{n}), D(\theta)\in\mathcal{S}(2\mathfrak{n}), \sigma\in\mathbb{C}$. In our application, we take $B(\theta)=B_{Kep}(e, \theta)$, the system (\ref{1.9e}) can be regarded as a perturbation of system ($\ref{Kep}$).
The system (\ref{1.9e}) is called non-degenerate, if it has only trivial zero solution. When the system is non-degenerate for $\sigma=0$, it's nature
to ask whether we can give an estimation for the upper bound of $|\sigma |$ such that the non-degeneracy preserves. This is closely related to the stability of system (\ref{1.9e}).
In our present paper, by the trace formula,
we will estimate the relative Morse index and Maslov-type index which is important in the study of the stability. Moreover, we can give an estimation for the upper bound such that the non-degeneracy of system (\ref{1.9e}) preserves over the full range $e\in[0,1)$. Consequently, some new stability regions for the Lagrange solution, Euler solution and the regular $(1+n)$-gon ERE are given. Precisely, We obtain the following main theorems.

Denote by $A|_{E_{\mathfrak{n}}(-I)}=-J_{2\mathfrak{n}}\frac{d}{dt}$ with domain
$$
E_{\mathfrak{n}}(-I)=\left\{z\in W^{1,2}([0,T];\mathbb{C}^{2\mathfrak{n}})\ |\ z(0)=-z(T) \right\}.
$$
where $B, D$ are bounded linear operators defined by $(Bz)(t) = B(t)z(t), (Dz)(t) = D(t)z(t)$ on $E_{\mathfrak{n}}(-I)$. Then $A|_{E_{\mathfrak{n}}(-I)}$ is a
self-adjoint operator with compact resolvent, moreover for $\sigma\in \rho(A)$, the resolvent
set of $A|_{E_{\mathfrak{n}}(-I)}$, $(A|_{E_{\mathfrak{n}}(-I)}-\sigma I_{2\mathfrak{n}})^{-1}$ is Hilbert-Schmidt. Our first theorem gives an estimation
of the relative Morse index $\mathcal{I}(A|_{E_{2}(-I)}-B_{Kep}, A|_{E_{2}(-I)}-B_{Kep}-\sigma D)$ and the Maslov-type index $i_{\omega}(\gamma)$ with $\omega=-1$. The relative Morse index and Maslov-type index  will be briefly introduced in Section \ref{sec4.1}.
%
%
%

\begin{thm}\label{thm1.1}
        If $D(A|_{E_{2}(-I_{2})}-B_{Kep})^{-1}$ has only real eigenvalues on $E_{2}(-I_{2})$, then for $(\sigma,e)$ satisfies
    \bea
|\sigma|<\frac{1}{\sqrt{f(e)}}, \nonumber
    \eea
we have $\mathcal{I}(A|_{E_{2}(-I_{2})}-B_{Kep},A|_{E_{2}(-I_{2})}-B_{Kep}-\sigma D)=0$ and $i_{-1}(\gamma_{\sigma,e})=2$, where $f(e)$ defined by (\ref{eq4.19}) is a specific smooth function with respect to $e\in [0,1)$.
Moreover, if for such $(\sigma,e)$, $i_1(\gamma_{\sigma,e})=0$ holds, then $\gamma_{\sigma,e}(T)$ is linear stable and
$$
\gamma_{\sigma,e}(T)\approx R(\vartheta_{1})\diamond R(\vartheta_{2})\ \ \text{with}\ \ \vartheta_{1}, \vartheta_{2}\in(\pi,2\pi),
$$
where $\approx$ denotes the symplectic similarity.
\end{thm}

As applications, we will give specific formulas for the stable regions of the Lagrange solution and Euler solution.
\begin{thm}\label{thm1.2}
    Lagrange solution is linearly stable if
    \bea
\beta_{L}<9-(3-\frac{1}{\sqrt{f_{L}(e)}})^2,\  \forall e\in [0,1),\nonumber
    \eea
and Euler solution is elliptic-hyperbolic if
    \bea
\beta_{E}<\frac{1}{2\sqrt{f_{L}(e)}},\  \forall e\in [0,1),\nonumber
    \eea
    where $f_{L}(e)$ is an explicit function with respect to $e$, defined by (\ref{eqflag}).
    The following figure gives the estimation derived using these theorems.
\begin{figure}[H]
    \centering
    \begin{subfigure}[b]{0.41\textwidth}
        \includegraphics[height=7cm,width=\textwidth]{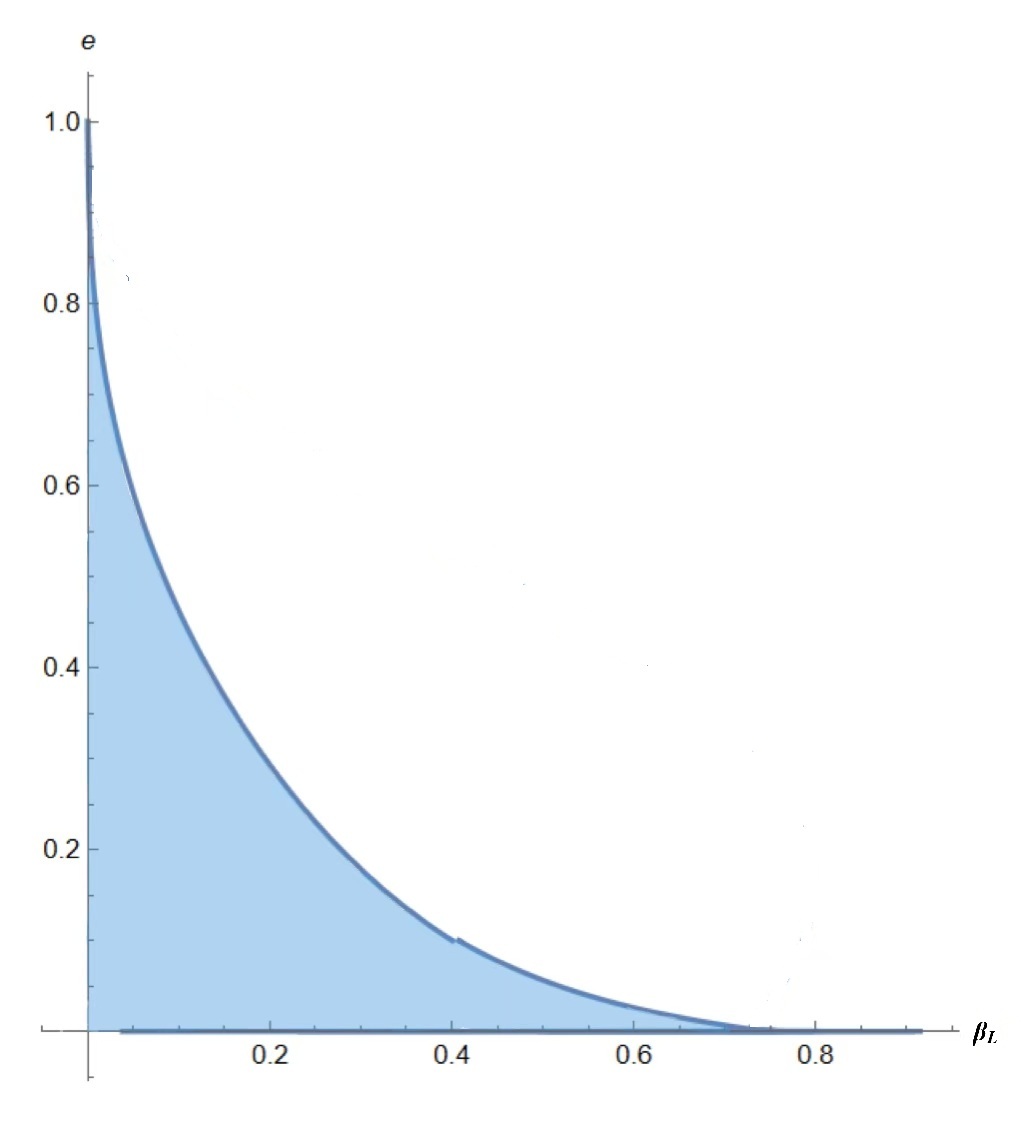}
        \caption{Estimation of the linear stability region for Lagrange solution (blue area).}
        \label{fig:sub1}
    \end{subfigure}
    \hspace{0.02\textwidth}
    \begin{subfigure}[b]{0.42\textwidth}
        \includegraphics[width=\textwidth]{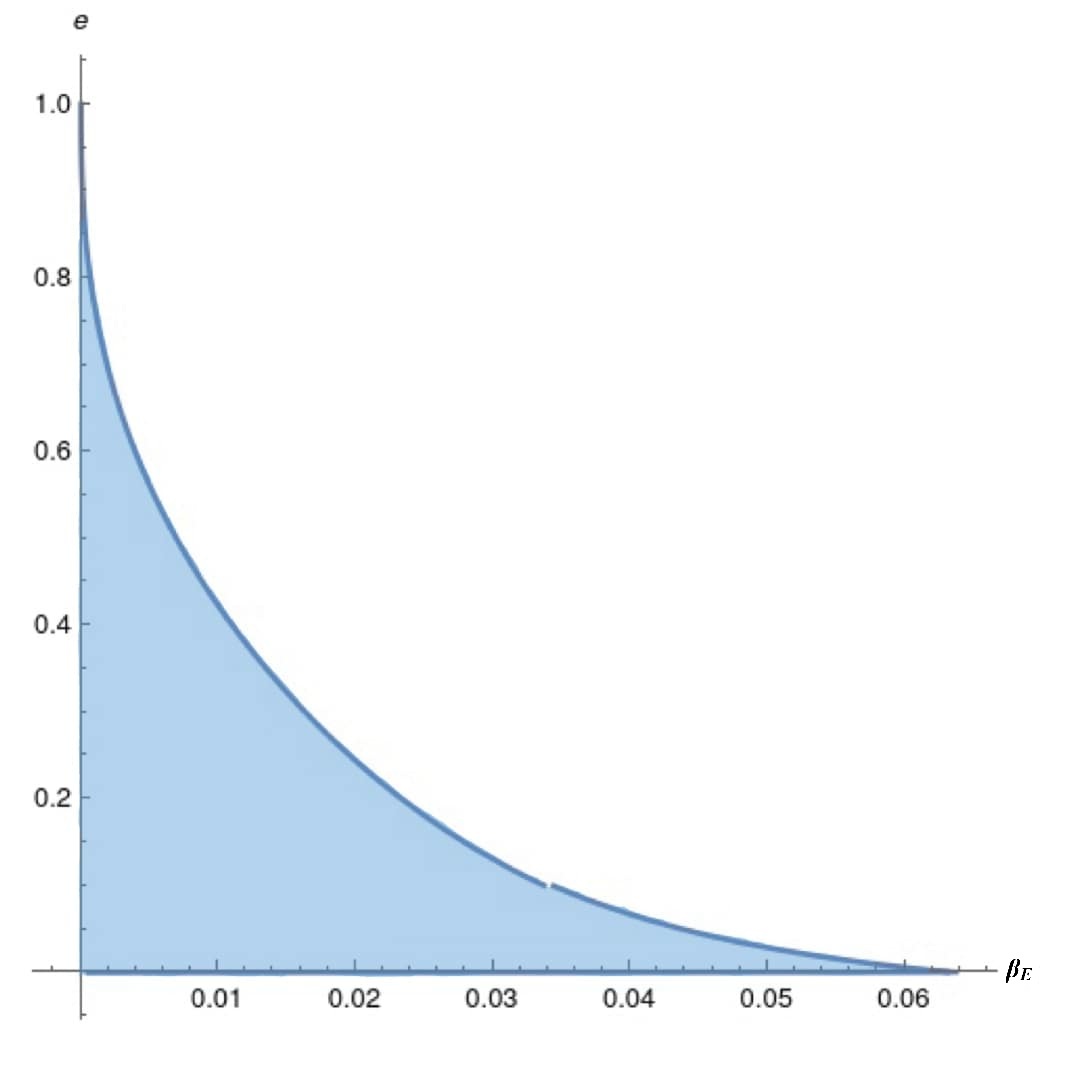}
        \caption{Estimation of the elliptic-hyperbolic region for Euler solution (blue area).}
        \label{fig:sub2}
    \end{subfigure}
    \label{fig:total}
\end{figure}
\end{thm}
For the $N$-body ERE with $N>3$, the dimension of the system is higher. In many cases, this system can be controlled by a series of simplified systems, which we refer to as $(\alpha,\eta)$-type system, as follows,
$$
\dot{\gamma}_{\alpha,\eta,e}(\theta)=JB_{\alpha,\eta}\gamma_{\alpha,\eta,e}(\theta),\quad \gamma_{\alpha,\eta,e}(0)=I_4,
$$
where
$$
B_{\alpha,\eta}(\theta)=\left(\begin{array}{cc}
    I_2 & -J_2 \\
    J_2 & I_2-\frac{R_{\alpha,\eta}}{1+e\cos\theta}
\end{array}\right),\ \ R_{\alpha,\eta}=\alpha I_2+\eta \tilde{N},
$$
with $\alpha\geq 1$, $\eta\geq 0$ and $\tilde{N}=\left(\begin{array}{cc}
    1 & 0 \\
    0 & -1
\end{array}\right)$.
For this system, we obtain following theorem.
\begin{thm}\label{thm1.5}
    For the $(\alpha,\eta)$-type system (\ref{eq5.25}) with $\alpha\geq1,\eta\geq0$, let $\gamma_{\alpha,\eta,e}(\theta)$ be the fundamental solution, then for $(\alpha,\eta,e)$ satisfies
    \bea
|\zeta(\alpha,\eta)|<\frac{1}{\sqrt{\tilde{f}(e)}},\nonumber
    \eea
we have $i_{-1}(\gamma_{\alpha,\eta,e})=2$, where $\zeta(\alpha,\eta)=\max\{|3-(\alpha+\eta)|,\ |\eta-\alpha|\}$ and $\tilde{f}(e)$ is defined in (\ref{f-}) is a specific smooth function with respect to $e\in [0,1)$. Moreover, if for such $(\alpha,\eta, e)$, $i_1(\gamma_{\alpha,\eta,e})=0$ holds, then $\gamma_{\alpha,\eta,e}(T)$ is linearly stable.
\end{thm}
As an application of above theorem, we give a stability result for the regular $(1+n)$-gon ERE.
\begin{thm}\label{thm1.4}
   For $n\geq 9$, the regular $(1+n)$-gon central configuration with center mass $m$ and eccentricity $e$ is linearly stable if
   \bea
(m,e)\in \mathcal{EE}_{(1+n)}, \nonumber
   \eea
   where $\mathcal{EE}_{(1+n)}$ is an explicit region in plane $(m,e)$ determined by $n$, defined by (\ref{EE1+n}).
\end{thm}

This paper is organized as following: In Section $2$, we introduce the $\mathbb{Z}_2$ decomposition of the $\mathcal{N}$-reversible Hamiltonian system, the trace formula with Lagrangian boundary condition and the index theory. Based on the $\mathbb{Z}_2$ decomposition and trace formula, we obtain a useful stability criteria in Theorem \ref{thm3.3}. In Section $3$, we first state the Meyer-Schmidt reduction for ERE which is useful in our study of the linear stability. Further investigation reveals that the linearized systems of many important EREs can be regarded as perturbed systems derived from the linearized Kepler problem, this motivates us to calculate the fundamental solution of linear system at Keplerian orbit. At the end of this section, we present the proof of  our main Theorem \ref{thm1.1} which give a quantitative analysis to the stable region of ERE via the stability criteria in Section $2$. In Section $4$, we apply the results in Section $3$ to Lagrange solution and Euler solution to estimate their stable region or elliptic-hyperbolic region. Moreover, for the $N$-body ERE with $N>3$, the system has high dimension, we introduce the $(\alpha,\eta)$-type system which is useful in the stability analysis for high dimension. In particular we give the stable region for the regular $(1+n)$-gon solution. In the last Section $5$, we give the details of some complex calculations in Section $4$.

\section{Trace formula, index theory and stability criteria}

In this section, we will first introduce the $\mathbb{Z}_2$ decomposition of the $\mathcal{N}$-reversible Hamiltonian system. We will see that, for the $\mathcal{N}$-reversible system, it is natural to obtain the Lagrangian boundary conditions. Secondly, we will briefly introduce the trace formula with Lagrangian boundary condition, the relative Morse index and
Maslov-type index theory. By combining $\mathbb{Z}_2$ decomposition and the trace formula, a useful stability criterion can be derived, as stated in Theorem \ref{thm3.3}.

\subsection{$\mathbb{Z}_2$ decomposition of the $\mathcal{N}$-reversible Hamiltonian system}\label{sec2}

The $\mathcal{N}$-reversible symmetry comes from the reversible system, more details can
be found in \cite{D77}.
In the study of the periodic solutions in the $N$-body problem, in many cases, the $\mathcal{N}$-reversible symmetry naturally appears. More precisely,
its Hamiltonian function $H(x)$ has the symmetry $H(\mathcal{N}x)=H(x)$ for some anti-symplectic orthogonal matrix $\mathcal{N}$ (\,i.e.,
$\mathcal{N}J=-J\mathcal{N},\mathcal{N} \in O(2n)$\,) with $\mathcal{N}^2=I$ , and we are interested in studying the periodic solution $x(t)$ satisfies
\bea
x'(t)=J\nabla H(x), x(t)=\mathcal{N}x(t), x(T+t)=x(t).\label{N-Hilm}
\eea
Because of the symmetry, the
linearized Hamiltonian equation of (\ref{N-Hilm}) at periodic solution $x(t)$ also has reversible symmetry
with respect to $\mathcal{N}$,
\bea
z'(t)=JB(t)z(t),\ \ B(t)=D^{2}H(x(t)),\ \ \text{with}\ \ \mathcal{N}B(t)=B(T-t)\mathcal{N}.\label{N-Hilm2}
\eea
In particular, in our study of the linear stability of the ERE, it will be shown in Section \ref{Sec2} that
\bea
B(\theta)=\left( \begin{array}{cccc} I_{k} & -\mathbb{J}_{k/2} \\
\mathbb{J}_{k/2} & I_{k}-\frac{\mathcal{R}}{1+e\cos\theta}
\end{array}\right),\quad \theta\in[0,2\pi], \nonumber \eea
where $k=2N-4$ and $e$ is the eccentricity, $\mathcal{R}=I_{k}+\mathcal{D}$ with
$\mathcal{D}=\frac{1}{\lambda}\mathbb{A}^\mathcal{T}D^2U(a)\mathbb{A}\big|_{w\in\mathbb{R}^{k}}$,
      and $\lambda=\frac{U(a)}{I(a)}$ which depends on the central configuration $a$.
For ERE, the useful case is
\bea
\mathcal{N}=diag\{\tilde{N},-\tilde{N}\}, \tilde{N}=\mathrm{diag}\{I,-I\},\label{Sp-symm}
\eea
Then, the $\mathcal{N}$-reversible symmetry (\ref{N-Hilm2}) with (\ref{Sp-symm}) for ERE is equivalent to
\bea
\tilde{N}\mathcal{D}\tilde{N}=\mathcal{D}.\label{symm2}
\eea
\begin{rem}
One can check that the Lagrange solution and Euler solution always satisfy the symmetry (\ref{symm2}), see (\ref{eq5.2}) in Section $4.1$ and (\ref{BE}) in
Section $4.2$. Hence, we can use the $\mathbb{Z}_{2}$ decomposition to these planar $3$-body EREs directly. However,  not all the EREs satisfy the symmetry (\ref{symm2}),
for example, the regular $(1+n)$-gon solution, see (\ref{2.17}) in Section $4.4$, so we can not use the $\mathbb{Z}_{2}$ decomposition directly. Instead, we introduce
a simple $(\alpha,\eta)$-type system (\ref{apl-eta}) which has $\mathbb{Z}_{2}$ decomposition, and use it to control the regular $(1+n)$-gon system and estimate the stable region.
More details can be found in Section $4.3$.
\end{rem}
In general, for the $2\mathfrak{n}$-dimensional linear Hamiltonian system with the symmetry (\ref{N-Hilm2}), we can do the $\mathbb{Z}_2$ decomposition  in the following.
Recall that the operator
$A|_{E_{\mathfrak{n}}(S)}=-J\frac{d}{dt}$ with domain
$$
E_{\mathfrak{n}}(S)=\left\{z\in W^{1,2}([0,T];\mathbb{C}^{2\mathfrak{n}})\ |\ z(0)=Sz(T) \right\}.
$$
Let $S$ satisfy
$
\mathcal{N}S^\mathcal{T}=S\mathcal{N}.
$
Define
$$
g: E_{\mathfrak{n}}(S) \rightarrow E_{\mathfrak{n}}(S)
\ \
\text{by}\ \
z(t) \mapsto \mathcal{N}z(T-t). \nonumber
$$
Obviously we have $(g^2z)(t)=z(t)$, and $g$ generates a $\mathbb{Z}_2$ group action on $E_{\mathfrak{n}}(S)$. Then we have
$$E_{\mathfrak{n}}(S)=\ker(g-id)\oplus \ker(g+id).$$
Denote by $E_{\mathfrak{n}}^\pm(S)=\ker(g\pm id)$, where $id$ is the identity mapping. For $\forall z\in E_{\mathfrak{n}}^+(S)$, we have $gz=z$, i.e.\ $\mathcal{N}z(T-t)=z(t)$, which implies that
\be
\mathcal{N}z(\frac{T}{2})=z(\frac{T}{2}),\quad S\mathcal{N}z(0)=z(0).\nonumber
\ee
Similar, for $\forall z\in E_{\mathfrak{n}}^-(S)$, we have
\be
\mathcal{N}z(\frac{T}{2})=-z(\frac{T}{2}),\quad S\mathcal{N}z(0)=-z(0).\nonumber
\ee
Denoted by $V^+(S\mathcal{N})$ and $V^-(S\mathcal{N})$ the positive and negative definite subspaces of $S\mathcal{N}$ respectively, obviously they are Lagrangian subspaces of $\mathbb{R}^{2\mathfrak{n}}$. Let
\bea
 \tilde{E}_{\mathfrak{n}}^{\pm}(S)=\{z\in W^{1,2}([0,\frac{T}{2}];\mathbb{C}^{2\mathfrak{n}})\, |\, z(0)\in V^{\pm}(S\mathcal{N}), z(\frac{T}{2})\in V^{\pm}(\mathcal{N})\}, \nonumber
\eea
then $E_{\mathfrak{n}}^\pm(S)$ are isomorphic to $\tilde{E}_{\mathfrak{n}}^\pm(S)$. Especially, when $S=-I_{\mathfrak{n}}$ and $\mathcal{N}=diag\{\tilde{N},-\tilde{N}\}$
with $\tilde{N}=\mathrm{diag}\{I_{\mathfrak{n}/2},-I_{\mathfrak{n}/2}\}$, we have
\be
\begin{aligned}
\tilde{E}_{\mathfrak{n}}^+(-I)=\{z\in W^{1,2}([0,\frac{T}{2}];\ \mathbb{C}^{2\mathfrak{n}})\  |\  z(0)\in V^+(-\mathcal{N}),\ z(\frac{T}{2})\in V^+(\mathcal{N})\}, \\
\tilde{E}_{\mathfrak{n}}^-(-I)=\{z\in W^{1,2}([0,\frac{T}{2}];\ \mathbb{C}^{2\mathfrak{n}})\  |\  z(0)\in V^-(-\mathcal{N}),\ z(\frac{T}{2})\in V^-(\mathcal{N})\}. \nonumber
\end{aligned}
\ee
If $B(t)$  satisfies (\ref{N-Hilm2}), that is $\mathcal{N}B(t)\mathcal{N}=B(T-t)$ holds, then $(A-B)(gz)=g((A-B)z)$ and we have
\be
\begin{aligned}
g\left(\left(A-B\right)\left(gz\right)\right)=\left(A-B\right)\left(gz\right),\  \forall z\in E_{\mathfrak{n}}^+(S), \ \
g\left(\left(A-B\right)\left(gz\right)\right)=-\left(A-B\right)\left(gz\right),\  \forall z\in E_{\mathfrak{n}}^-(S), \nonumber
\end{aligned}
\ee
which means that $A-B$ is invariant on $E_{\mathfrak{n}}^{\pm}(S)$, then we have
\be
(A-B)|_{E_{\mathfrak{n}}(S)}=(A-B)|_{E_{\mathfrak{n}}^+(S)}\oplus (A-B)|_{E_{\mathfrak{n}}^-(S)} \nonumber
\ee
and hence
$$
\ker\left(\left(A-B\right)|_{E_{\mathfrak{n}}(S)}\right)=\ker\left(\left(A-B\right)|_{E_{\mathfrak{n}}^+(S)}\right)\oplus \ker\left(\left(A-B\right)|_{E_{\mathfrak{n}}^-(S)}\right).
$$
Since $E_{\mathfrak{n}}^\pm(S)$ are isomorphic to $\tilde{E}_{\mathfrak{n}}^\pm(S)$, we have
$$
\dim\ker\left(\left(A-B\right)|_{E_{\mathfrak{n}}^\pm(S)}\right)=\dim\ker\left(\left(A-B\right)|_{\tilde{E}_{\mathfrak{n}}^\pm(S)}\right).
$$
From the above analysis, the linear Hamiltonian system with $S$-periodic boundary conditions are closely related to the Lagrangian boundary conditions.
When it has the $\mathcal{N}$-reversible symmetry, the space $E_{\mathfrak{n}}(S)$ can be decomposed into subspaces $\tilde{E}_{\mathfrak{n}}^{\pm}(S)$ which have the
Lagrangian boundary condition. In the next sections, we will study the Hamiltonian system with Lagrangian boundary condition and it's trace formula and index theory,
which is the main tool in our study of the stability problem in ERE.

\subsection{Trace formula with Lagrangian boundary condition}


Consider the standard symplectic space $(\mathbb{C}^{2\mathfrak{n}}, \omega)$, where the standard symplectic structure $\omega(x,y)$ in $\mathbb{C}^{2\mathfrak{n}}$ is given by
$$
\omega(x,y)=<J_{2\mathfrak{n}}x,y>,
$$
where $<\cdot,\cdot>$ is the standard Hermitian inner product. A Lagrangian subspace $V$ of $(\mathbb{C}^{2\mathfrak{n}},\omega)$ is an isotropic subspace of dimension $\mathfrak{n}$, that is, for
any $x, y\in V$, $\omega(x, y)=0$. Denote by $Lag(\mathbb{C}^{2\mathfrak{n}},\omega_{n})$ the set of Lagrangian subspaces of $\mathbb{C}^{2\mathfrak{n}}$.
Suppose $V\in Lag(\mathbb{C}^{2\mathfrak{n}},\omega_{n})$, a Lagrangian frame for $V$ is a linear map $Z: \mathbb{C}^{\mathfrak{n}}\rightarrow \mathbb{C}^{2\mathfrak{n}}$ whose image is $V$.
Let $V_{0},\ V_{1}$ be two Lagrangian subspaces of $\mathbb{C}^{2\mathfrak{n}}$, $B,D\in C([0,T],\mathcal{S}(2\mathfrak{n}))$. The eigenvalue problem of linear Hamiltonian system with Lagrangian boundary condition is to find $\sigma\in \mathbb{C}$ such that
\be{\label{eq3.1}}
z'(t)=J(B(t)+\sigma D(t))z(t),
\ \
z(0)\in V_{0},\ z(T)\in V_{1},
\ee
has non-trivial solutions in domain
\be
E_\mathfrak{n}(V_{0},V_{1})=\{z\in W^{1,2}([0,T];\mathbb{C}^{2\mathfrak{n}})\ \ |\ \ z(0)\in V_{0}, z(T)\in V_{1}\}.\nonumber
\ee
 Let $B,D$ be the operatoers on $E(V_{0},V_{1})$, and $(Bz)(t)=B(t)z(t)$, $(Dz)(t)=D(t)z(t)$. Then (\ref{eq3.1}) can be written as $Az=Bz+\sigma Dz$ on $E_\mathfrak{n}(V_{0},V_{1})$. In the case $A|_{E_\mathfrak{n}(V_{0},V_{1})}-B$ is non-degenerate, it can be written as $D(A|_{E_\mathfrak{n}(V_{0},V_{1})}-B)^{-1}z=\sigma^{-1}z$.

Let
\be
\mathcal{F}(B,D;E_{\mathfrak{n}}(V_{0},V_{1}))=D(A|_{E_{\mathfrak{n}}(V_{0},V_{1})}-B)^{-1}, \nonumber
\ee
for convenience, denote by $\mathcal{F}=\mathcal{F}(B,D)=\mathcal{F}(B,D;E_{\mathfrak{n}}(V_{0},V_{1}))$, if there is no confusion.
Then the eigenvalue problem of (\ref{eq3.1}) can be transformed to the normal eigenvalue problem of operator $\mathcal{F}$ on $E_\mathfrak{n}(V_{0},V_{1})$.

In \cite{HOW19}, they give the trace formula to calculate the trace of $\mathcal{F}^{m}$, following the notations in \cite{HOW19},
let $\gamma_{\sigma}(t)$ be the fundamental solution of (\ref{eq3.1})
and
\bea\label{eq3.3}
P=(Z_0,\gamma_0^{-1}(T)Z_1), \quad Q_d=(Z_0,O_{2\mathfrak{n}\times \mathfrak{n}}),\ \ G_j=P^{-1}M_jQ_d,
\eea
where $Z_0$ and $Z_1$ are frames of $V_0$ and $V_1$, respectively, and
\bea\label{eq3.5}
M_j=\int_{0}^{T}J\hat{D}(t_1)\int_{0}^{t_1}J\hat{D}(t_2)\cdots \int_{0}^{t_j-1}J\hat{D}(t_j) dt_j\cdots dt_2dt_1,\ \ \hat{D}(t)=\gamma_{0}^\mathcal{T}(t)D(t)\gamma_0(t),
\eea
Then we can describe the trace formula in the following.
\begin{thm}\label{th3.1}(See Theorem 3.8 of \cite{HOW19})
    With the above notations, we have that
    \be
   Tr(\mathcal{F}^m)=m \sum_{k=1}^m\frac{(-1)^k}{k}(\sum_{j_1+\cdots +j_k=m}Tr(G_{j_1}\cdots G_{j_k})).\nonumber
   \ee
In particular, for $m=1,2$, we have
\be\label{2orTrace}
Tr(\mathcal{F})=-Tr(G_1),\quad Tr(\mathcal{F}^2)=Tr(G_1^2)-2Tr(G_2).
\ee
\end{thm}
For $m\geq 2$, $\mathcal{F}$ is trace class operator and have
\be\label{eq3.7}
Tr(\mathcal{F}^m)=\sum_j\frac{1}{\lambda_j^m},
\ee
where $\lambda_j$'s are nonzero eigenvalues of the system (\ref{eq3.1}), and each $\lambda_j$ appears as many times as its multiplicity. More details of the trace formula can be found in \cite{HOW19}.

Now we turn to the view of Lagrangian system. Consider the Sturm-Liouville system
\bea\label{eqst3.8}
-(\mathcal{P} y'+\mathcal{Q} y)^{'}+\mathcal{Q}^\mathcal{T} y'+\left(\mathcal{R}_0+\sigma \mathcal{R}_1\right) y=0,
\eea
where $\mathcal{Q}$ is a continuous path of $\mathfrak{n} \times \mathfrak{n}$ real matrices, and $\mathcal{P}, \mathcal{R}_0, \mathcal{R}_1$ are continuous paths of $\mathfrak{n} \times \mathfrak{n}$ real symmetric matrices on $[0, T]$, $\sigma\in\mathbb{C}$. Instead of Legendre convexity condition, we assume that for any $t \in[0, T], \mathcal{P}(t)$ is invertible. Set $x=\mathcal{P} y'+\mathcal{Q} y,\ z=(x, y)^\mathcal{T}$, and the boundary condition is given by
$$
z(0) \in V_0, \quad z(T) \in V_1 .
$$ Then (\ref{eqst3.8}) corresponds to the Hamiltonian system
\bea
z'=J B_\sigma(t) z, \quad z(0) \in V_0, \quad z(T) \in V_1,\nonumber
\eea
with
\bea
B_\sigma(t)=\left(\begin{array}{cc}
\mathcal{P}^{-1}(t) & -\mathcal{P}^{-1}(t) \mathcal{Q}(t) \\
-\mathcal{Q}(t)^\mathcal{T} \mathcal{P}^{-1}(t) & \mathcal{Q}(t)^\mathcal{T} \mathcal{P}^{-1}(t) \mathcal{Q}(t)-\mathcal{R}(t)-\sigma \mathcal{R}_1(t)
\end{array}\right) .\nonumber
\eea

Denote by
\bea
\mathcal{A}_0=-\frac{d}{d t}\left(\mathcal{P} \frac{d}{d t}+\mathcal{Q}\right)+\mathcal{Q}^\mathcal{T} \frac{d}{d t}+\mathcal{R}_0,\nonumber
\eea
which is a self-adjoint operator on $L^2\left([0, T], \mathbb{C}^{\mathfrak{n}}\right)$ with domain
$$
\mathbb{D}_\mathfrak{n}(V_0, V_1)=\left\{y \in W^{2,2}\left([0, T]; \mathbb{C}^\mathfrak{n}\right), z(0) \in V_0, z(T) \in V_1\right\}.
$$
Then the relationship between eigenvalues of  Lagrangian system and Hamiltonian system can be stated as the following Proposition. Denote by $\Pi (\cdot)$ the set of eigenvalues of an operator.
\begin{prop}\label{propltoh}(See Corollary 3.9 of \cite{HOW19})
    Under the notations above,
\bea
\operatorname{det}\left(I+\mathcal{F}\left(B_0, D, E_{\mathfrak{n}}(V_{0}, V_{1})\right)\right)=\operatorname{det}\left(I+\mathcal{R}_1 \mathcal{A}_0^{-1}|_{D_{\mathfrak{n}}(V_{0},V_{1})}\right),\nonumber
\eea
consequently, $\Pi\left(\mathcal{F}\left(B_0, D, E_{\mathfrak{n}}(V_{0}, V_{1})\right)\right)=\Pi\left(\mathcal{R}_1 \mathcal{A}_0^{-1}|_{D_{\mathfrak{n}}(V_{0}, V_{1})}\right)$ with the same multiplicity.
\end{prop}
\begin{rem}\label{detm}
Proposition \ref{propltoh} is also true for the $S$-periodic boundary condition, see \cite{HOW15}.
\end{rem}

\subsection{Index theory and stability criteria via trace formula}\label{sec4.1}

Let $\{\mathcal{O}(s), s\in [0,1]\}$ be a continuous path of self-adjoint Fredholm operators on a Hilbert space $\mathcal{H}$. The spectral flow of path $\{\mathcal{O}(s), s\in [0,1]\}$ denoted by $\mathrm{Sf}(\{\mathcal{O}(s), s\in [0,1]\})$ counts the net change in the number of negative eigenvalues of $\mathcal{O}(s)$ as $s$ goes from $0$ to $1$, where the enumeration follows from the rule that each negative eigenvalue crossing to the positive axis contributes $+1$ and each positive eigenvalue crossing to the negative axis contributes $-1$, and for each crossing, the multiplicity of eigenvalue is counted.

For Hamiltonian systems, let $\mathcal{O}(s)=A-B_s$, $s\in [0,1]$, where $B_s\in C([0,T],\mathcal{S}(2\mathfrak{n}))$. We can define the relative Morse index of $A-B_0$ and $A-B_1$ as
\be
\mathcal{I}(A-B_0,A-B_1)=-\mathrm{Sf}(\{A-B_s,s\in [0,1]\}).
\ee

It can be roughly understood as measuring the dimension difference of maximal negative definite subspaces of $A-B_0$ and $A-B_1$ as $s$ goes form $0$ to $1$. Usually, the dimensions might be infinite, but $\mathcal{I}(A-B_0,A-B_1)$ could be finite.

 We list some fundamental properties of relative Morse index here, for details please refer to \cite{HOW15}.
\begin{prop}\label{prop3.2}
    (1) For $B_0$, $B_1$, $B_2$, then
    \be
\mathcal{I}(A-B_0,A-B_1)+\mathcal{I}(A-B_1,A-B_2)=\mathcal{I}(A-B_0,A-B_2);\nonumber
    \ee
    (2) Let $D=B_1-B_0$ and $B_s=B_0+sD$, let $\kappa =\{s_0\in [0,1] | \ker(A-B_{s_0})\neq 0\}$, then
    \be
\mathcal{I}(A-B_0,A-B_1)\leq \sum_{s_0\in\kappa} \dim(\ker(A-B_{s_0})); \nonumber
    \ee
    (3) Suppose $D_1\leq D\leq D_2$, then
    \be
\mathcal{I}(A-B,A-B-D_1)\leq \mathcal{I}(A-B,A-B-D)\leq \mathcal{I}(A-B,A-B-D_2). \nonumber
    \ee
\end{prop}
On the other hand, Maslov-type index theory serves as an important tool for studying stability.  
We give a briefly review the Maslov-type index,  the detail  could be  found  in \cite{L02}.
Let $\Sp(2\mathfrak{n})$ be the set of $2\mathfrak{n}\times2\mathfrak{n}$ real symplectic matrix. For $\tau>0$ we are interested in paths in $\Sp(2\mathfrak{n})$:
\begin{equation} \label{sym path}
\mathcal{P}_{\tau}(2\mathfrak{n})=\left\{\gamma \in C([0, \tau], \operatorname{Sp}(2\mathfrak{n})) \mid \gamma(0)=I_{2\mathfrak{n}}\right\}. \nonumber
\end{equation}
For any $\omega\in\mathbb{U}$, the following  $\omega$-degenerate hypersurface of codimension one in
 $\Sp(2\mathfrak{n})$ is defined  \cite{L02}:
$$\Sp(2\mathfrak{n})_\omega^0=\{M\in \Sp(2\mathfrak{n}) | \det(M-\omega I_{2\mathfrak{n}})=0 \}.   $$
Moreover, the $\omega$-regular set of $\Sp(2\mathfrak{n})$ is defined by $\Sp(2\mathfrak{n})^*_{2\mathfrak{n}}=\Sp(2n)\setminus \Sp(2\mathfrak{n})_\omega^0$.

For $M\in \Sp(2\mathfrak{n})_\omega^0$, we define a co-orientation of  $\Sp(2\mathfrak{n})_\omega^0$
at $M$ by the positive direction $\frac{d}{dt}Me^{tJ}|_{t=0}$ of the path $Me^{tJ}$ with $|t|$ sufficiently small.
Now will give the definition of $\omega$-index \cite{L02}.

\begin{defi}\label{def:Maslov-type index} For $\omega\in \mathbb{U}$, $\gamma\in \mathcal{P}_{\tau}(2\mathfrak{n})$, \textbf{the $\omega$-index} of $\gamma$  can be defined as
$$i_\omega(\gamma)= \begin{cases}[e^{-\epsilon J}\gamma: \Sp(2\mathfrak{n})_\omega^0]-\mathfrak{n}, & \text { if } \omega=1 \\ [e^{-\epsilon J}\gamma: \Sp(2\mathfrak{n})_\omega^0], & \text { if } \omega\neq1.\end{cases}
 $$ For $\epsilon>0$ small enough, where $[\cdot : \cdot]$ is the intersection number. We also denote the nullity of $\gamma$ by
 $$\nu_\omega(\gamma)=\dim_{\mathbb{C}}\ker_{\mathbb{C}}(\gamma(\tau)-\omega I).$$
 \end{defi}
The following symplectic matrices were introduced as basic normal forms:
\begin{equation}\label{basic normal forms}
D(\lambda)=\left(\begin{array}{cc}
\lambda & 0 \\
0 & \lambda^{-1}
\end{array}\right),\ \
N_{1}(\lambda, a)=\left(\begin{array}{cc}
\lambda & a \\
0 & \lambda
\end{array}\right), \ \
R(\vartheta)=\left(\begin{array}{cc}
\cos \vartheta & -\sin \vartheta \\
\sin \vartheta & \cos \vartheta
\end{array}\right),\nonumber
\end{equation}
\begin{equation}\label{basic normal forms2}
N_{2}(e^{\sqrt{-1}\vartheta},b)=\left(\begin{array}{cc}
R(\vartheta) & b \\
0 & R(\vartheta)
\end{array}\right),\ \ \text{with} \ \
b=\left(\begin{array}{cc}
b_{1} & b_{2} \\
b_{3} & b_{4}
\end{array}\right),
\nonumber
\end{equation}
where $\lambda\in \mathbb{R}\setminus\{0\}$, $\vartheta\in(0,\pi)\cup(\pi,2\pi)$, $a=\pm1,0$ and $b_{i}\in\mathbb{R}, b_{1}\neq b_{2}$.

Let $\Omega_{0}(M)$ be the path-connected component containing $M=\gamma(\tau)$ of the set
$$
\Omega(M)=\{L\in\Sp(2\mathfrak{n})\,|\,\sigma(N)\cap\mathbb{U}=\sigma(M)\cap\mathbb{U}, v_{\lambda}(L)=v_{\lambda}(M), \forall \lambda\in\sigma(M)\cap\mathbb{U}\}
$$
where $\sigma(\cdot)$ denotes the spectrum of a matrix, that is the set of its total eigenvalues.  Here $\Omega_{0}(M)$
is called the homotopy component of $M$ in $\Sp(2\mathfrak{n})$. For a continuous family of paths $\ga_{s}(t)$ with
$(s,t)\in[0,1]\times[0,T], \ga_{s}(T)\in \Omega_{0}(\ga_{0}(T))$, then $i_\omega(\ga_s)$ is independent of $s$.
Then any $M\in\Sp(2\mathfrak{n})$ can be connected to $L$ in $\Omega_{0}(M)$, we denote it briefly as the symplectic similarity, where
$L=M_{1}\diamond \cdots \diamond M_{j}$
with $M_{i}, i =1,...,j$ in basic normal form.
The following Theorem is a stability criterion obtained through index theory.
\begin{thm}(See (9.3.3) on page 24 of \cite{L02} with $\omega=-1$)\label{thm3.5}
Let $\gamma(t)\in C([0,T],\mathrm{Sp}(2\mathfrak{n}))$ be a path of fundamental solution of linear Hamiltonian system, then
when $|i_{1}(\ga)-i_{-1}(\ga)|=\mathfrak{n}$, $\ga(T)$ is spectral stable. Moreover, if $i_{-1}(\ga)=\mathfrak{n}$, $i_{1}(\ga)=0$, then $\ga(T)$ is linear stable and
$\ga(T)$ is symplectic similar to $R(\vartheta_1)\diamond R(\vartheta_2)\diamond\cdots\diamond R(\vartheta_n)\ \ \text{for some}\ \ \vartheta_{i}\in(\pi,2\pi)$.
\end{thm}
Usually, $i_{\omega}(\gamma)$ is not easy to calculate, but we have the following Proposition, which establishes the connection between the relative
Morse index and the Maslov-type index.
\begin{prop}\label{prop3.6}
    Let $\gamma$ be the fundamental solution of (\ref{1.9e}), then
\bea
\mathcal{I}(A|_{E_\mathfrak{n}(S)},A|_{E_\mathfrak{n}(S)}-B)=\left\{\begin{aligned}
    &i_1(\gamma)+\mathfrak{n}, \quad S=I,\\
    &i_{-1}(\gamma), \quad\quad  S=-I.
\end{aligned}\right. \nonumber
\eea
\end{prop}
Now, we can give the following important stability criteria via the trace formula, this is useful in our applications.
\begin{thm}\label{thm3.3}
    Suppose $A-B$ is non-degenerate on its domain $E$ and $D(A-B)^{-1}$ has only real eigenvalues on $E$. If  $Tr(\mathcal{F}^2(B,D;E))<1$, then $A-B-\sigma D$ is non-degenerate on $E$ and
    \be
\mathcal{I}(A-B,A-B-\sigma D)=0,\ \ \text{for}\ \ \sigma\in[0,1]. \nonumber
    \ee
In particular, for $E=E_{\mathfrak{n}}(\pm I)$, we have
     \bea
i_1(\gamma_{\sigma})=i_1(\gamma_0),\quad i_{-1}(\gamma_{\sigma})=i_{-1}(\gamma_0). \nonumber
     \eea
     Morover, if $|i_1(\gamma_0)-i_{-1}(\gamma_0)|=\mathfrak{n}$ holds, then $\gamma_{\sigma}(T)$ is spectrally stable and if
     $i_1(\gamma_0)=0, i_{-1}(\gamma_0)=\mathfrak{n}$, then $\gamma_{\sigma}(T)$ is linearly stable and
$
\gamma_{\sigma}(T)\approx R(\vartheta_1)\diamond R(\vartheta_2)\diamond\cdots\diamond R(\vartheta_n)$ for some  $\vartheta_{i}\in(0,\pi)$.
\end{thm}
\begin{proof}
    Denote by $\{1/\lambda_j\}_{j\in \mathbb{Z}}$ the eigenvalues of $\mathcal{F}(B,D;E)$, then $\lambda_j\in \mathbb{R}$. From (\ref{eq3.7}),
    \bea
Tr(\mathcal{F}^2(B,D;E))=\sum_{j}\frac{1}{\lambda_j^2}>0.\nonumber
    \eea
    If  $Tr(\mathcal{F}^2(B,D;E))<1$, then for $\forall j$, we have $|\lambda_j|^2>1$, hence $A-B-\sigma D$ is non-degenerate on $E$ and $$\mathcal{I}(A-B,A-B-\sigma D)=0,\ \ \text{for}\ \ \sigma\in[0,1].$$
In particular, for $E=E_{\mathfrak{n}}(\pm I)$, from Proposition \ref{prop3.2} $(1)$ and Theorem \ref{thm3.3}, take $S=\pm I$, we have
     \bea
    \begin{aligned}
        \mathcal{I}(A|_{E_{\mathfrak{n}}(S)},A|_{E_{\mathfrak{n}}(S)}-B-\sigma D) =\mathcal{I}(A|_{E_{\mathfrak{n}}(S)},A|_{E_{\mathfrak{n}}(S)}-B)+\mathcal{I}(A|_{E_{\mathfrak{n}}(S)}-B,A|_{E_{\mathfrak{n}}(S)}-B-\sigma D)
        =\mathcal{I}(A|_{E_{\mathfrak{n}}(S)},A|_{E_{\mathfrak{n}}(S)}-B).
    \end{aligned} \nonumber
    \eea
    This combine with Proposition \ref{prop3.6}, we obtain
    \bea
        i_{1}(\gamma_{\sigma})&=&\mathcal{I}(A|_{E_{\mathfrak{n}}(I)},A|_{E_{\mathfrak{n}}(I)}-B-\sigma D)-\mathfrak{n}=\mathcal{I}(A|_{E_{\mathfrak{n}}(I)},A|_{E_{\mathfrak{n}}(I)}-B)
        -\mathfrak{n}=i_{1}(\gamma_0).\nonumber\\
         i_{-1}(\gamma_{\sigma})&=&\mathcal{I}(A|_{E_{\mathfrak{n}}(-I)},A|_{E_{\mathfrak{n}}(-I)}-B-\sigma D)
        =\mathcal{I}(A|_{E_{\mathfrak{n}}(-I)},A|_{E_{\mathfrak{n}}(-I)}-B)
        =i_{-1}(\gamma_0),\nonumber
    \eea
    By Theorem  \ref{thm3.5}, if $|i_1(\gamma_0)-i_{-1}(\gamma_0)|=n$, $\gamma_{\sigma}(T)$ is spectrally stable and if
     $i_1(\gamma_0)=0, i_{-1}(\gamma_0)=\mathfrak{n}$, then $\gamma_{\sigma}(T)$ is linearly stable for $\sigma\in[0,1]$
     and
$
\gamma_{\sigma}(T)\approx R(\vartheta_1)\diamond R(\vartheta_2)\diamond\cdots\diamond R(\vartheta_n)$ for some $\vartheta_{i}\in(\pi,2\pi)$.
\end{proof}
Especially for $D>0$ or $D<0$, $D(A-B)^{-1}$ is self-adjoint on $E$, hence its eigenvalues are real, then we have the following Corollary from Theorem \ref{thm3.3}.
\begin{cor}\label{thm3.4}
    Suppose $A-B$ is non-degenerate, $D > 0$ and $Tr(G_1^2)-$ $2 Tr\left(G_2\right)<1$ in domain $E$, then $A-B-\sigma D$ is non-degenerate, and $\mathcal{I}(A-B, A-B-\sigma D)=0$
for $\sigma\in[0,1]$. In particular, for $E=E_{\mathfrak{n}}(\pm I)$, we have
     \bea
i_1(\gamma_{\sigma})=i_1(\gamma_0),\quad i_{-1}(\gamma_{\sigma})=i_{-1}(\gamma_0). \nonumber
     \eea
     If $|i_1(\gamma_0)-i_{-1}(\gamma_0)|=\mathfrak{n}$ holds, then $\gamma_{\sigma}(T)$ is spectrally stable and if
     $i_1(\gamma_0)=0, i_{-1}(\gamma_0)=\mathfrak{n}$, then $\gamma_{\sigma}(T)$ is linearly stable
     and
$
\gamma_{\sigma}(T)$ is symplectic similar to $R(\vartheta_1)\diamond R(\vartheta_2)\diamond\cdots\diamond R(\vartheta_n)$ for some  $\vartheta_{i}\in(\pi,2\pi)$. Similar for the case $D<0$.
\end{cor}

\section{Quantitative analysis of ERE in planar $N$-body problem}
In this section, we first state the results of Meyer-Schmidt reduction for ERE, it's useful in stability studies. Then we will calculate the fundamental solution of linear Hamiltonian system at Keplerian orbit and prove the main Theorem \ref{thm1.1} which give a quantitative analysis to the stable region of ERE in planar $N$-body problem.
\subsection{Meyer-Schmidt reduction of ERE}\label{Sec2}
For the ERE in the planar $N$ body problem, from Meyer and Schmidt in \cite{MS05},
there are two four-dimensional invariant symplectic subspaces, $E_1$ and $E_2$
associated to the translation symmetry, dilation and rotation symmetry of the system and the remaining part $E_3$ such that $\mathbb{R}^{4N}=E_1\oplus E_2\oplus E_3$ . More precisely,
Meyer and Schmidt introduced a coordinate transformation $\mathbb{A}$ from the central configuration coordinates to the original coordinates, which depends on the central configuration. The linear transformation has the form $Q = \mathbb{A}X,  P = \mathbb{A}^{-\mathcal{T}} Y$ with
$X = (g, z,w)\in \mathbb{R}^2\times\mathbb{R}^2\times\mathbb{R}^{2N-4}$ and
$Y =(G,Z,W)\in\mathbb{R}^2\times\mathbb{R}^2\times\mathbb{R}^{2N-4}$, where
$\mathbb{A}\in GL(\R^{2N})$ satisfies
\bea  \mathbb{J}_{N}\mathbb{A} = \mathbb{A}\mathbb{J}_{N}, \qquad \mathbb{A}^\mathcal{T}\mathcal{M}\mathbb{A} = I_{2N},  \label{AA}\nonumber\eea
After this transformation, for the linearized Hamiltonian system $(\ref{ga})$, $\mathcal{B}(t)=D^{2}H(x(t))$ in this new coordinate system  has the form
$\mathcal{B}(t)=\mathcal{B}_1(t)\oplus \mathcal{B}_2(t)\oplus \mathcal{B}_{3}(t)$, where $\mathcal{B}_i(t)=\mathcal{B}|_{E_i}(t)$.
The essential part $\mathcal{B}_3(t)$ is a path of $(4N-8)\times(4N-8)$ symmetric matrices which is
closely related to the linear stability of the ERE.
By taking the rotating coordinates and using the true anomaly $\theta$ as the variables, the equation of the essential part is,
\bea \dot{\ga}(\theta)=JB(\theta)\ga(\theta),  \quad  \ga(0)=I_{2k}, \label{ga3e}
\eea
with
\bea \label{eq2.4e}
B(\theta)=\left( \begin{array}{cccc} I_{k} & -\mathbb{J}_{k/2} \\
\mathbb{J}_{k/2} & I_{k}-\frac{\mathcal{R}}{1+e\cos\theta}
\end{array}\right),\quad \theta\in[0,2\pi],  \eea
where $k=2N-4$ and $e$ is the eccentricity,
\bea \mathcal{R}=I_{k}+\mathcal{D},\ \ \text{with}\ \ \mathcal{D}=\frac{1}{\lambda}\mathbb{A}^\mathcal{T}D^2U(a)\mathbb{A}\big|_{w\in\mathbb{R}^{k}},
      \ \ \text{and}\ \ \lambda=\frac{U(a)}{I(a)}, \label{hess}
\eea
where $w$ denotes the essential part.
In Section \ref{sec4}, we will see some important EREs, where $\mathcal{R}$ has explicit expressions. For instance,
\bea
\mathcal{R}_{L}=\left(\begin{array}{cc}
    \frac{3+\sqrt{9-\beta_{L}}}{2} & 0 \\
    0 & \frac{3-\sqrt{9-\beta_{L}}}{2}
\end{array}\right),\ \
\mathcal{R}_{E}=\left(\begin{array}{cc}
    2\beta_{E}+3 & 0 \\
    0 & -\beta_{E}
\end{array}\right),\ \ \nonumber
\eea
\bea
\mathcal{R}_{M}=R_{1}\oplus\ldots R_{[\frac{n}{2}]},\ \ R_l=I+\frac{1}{\mu}\mathcal{U}(l), 1\leq l\leq[\frac{n}{2}].\nonumber
\eea
correspond to the Lagrange, Euler solutions and the regular $(1+n)$-gon solution, respectively, where $\beta_{L}\in[0,9]$, $\beta_{E}\in[0,7]$ and $\mathcal{U}(l)$, $\mu$ are given by (\ref{2.17})-(\ref{gon-n}) and (\ref{n-gon co}), which depend on the mass parameter $\beta_{M}=1/m$.

One can see that when $\beta_{L}=\beta_{E}=0$, Lagrange solution and Euler solution in (\ref{hess}) can be reduced into
\bea
B(\theta)=B_{kep}(\theta)=\left( \begin{array}{cccc} I_{2} & -J_{2} \\
J_{2} & I_{2}-\frac{\mathcal{R}_{kep}}{1+e\cos\theta}
\end{array}\right),\ \ \mathcal{R}_{kep}=\left( \begin{array}{cccc} 3 & 0 \\
0 & 0
\end{array}\right)\label{kep2}
\eea
Also, for the regular $(1+n)$-gon ERE, when
$\beta_{M}=0$ (\,i.e $m=+\infty$), $\hat{B}_i(\theta), 1\leq i\leq [\frac{n}{2}]$ which is given in (\ref{2.13}) can be decomposed into the form (\ref{kep2}).

Subsequently, we will see that the useful form (\ref{kep2}) is precisely the representation of the linearized Kepler system in central configuration coordinate.
Therefore, the Lagrange solution, Euler solution and the regular $(1+n)$-gon solution can be regarded as a perturbation of the linearized Kepler system. In order to study the
stability of the perturbed system of linearized Kepler system, we first calculate the fundamental solution of linear Hamiltonian system at Keplerian
orbit in the following section.
\subsection{Fundamental solution of Keplerian orbit}
Let $z(t)$ be a periodic solution of Hamiltonian system (\ref{eq H}), that is
\bea\label{eqs4.1}
    z'(t)=J\nabla H(z), \ \ z(0)=z(T),
\eea
the linearized equation of (\ref{eqs4.1}) at $z$ is
\bea\label{eql4.1}
y'(t)=JD^{2}H(z)y(t).
\eea
The following Lemma is useful in our calculation of the fundamental solution of Keplerian orbit.
\begin{lem}\label{lemfm}
    If $I(z)$ is a first integral of (\ref{eqs4.1}), then $J\nabla I(z)$ is a solution of (\ref{eql4.1}).
\end{lem}
\begin{proof}
    Suppose that $h^t$ and $f^s: \ E\mapsto\ E$ satisfying
    \bea
\frac{d}{d t}h^t z|_{t=0}=J\nabla H(z),\ \
\frac{d}{d s}f^sz|_{s=0}=J\nabla I(z),\quad \text{for}\quad \forall z\in E,\nonumber
    \eea
are Hamiltonian flows with respect to $J\nabla H$ and $J\nabla I$, respectively.
Since $I$ is a first integral of (\ref{eqs4.1}), then the two flows are commutable, that is, $f^s(h^t z)=h^t(f^s z)$ for $\forall z\in E$. Then we have
\bea\label{eqfh}
\frac{d}{d t}\left(f^s\left(h^t z\right)\right)=J\nabla H\left(f^s\left(h^t z\right)\right),
\eea
differentiate (\ref{eqfh}) with respect to $s$, we have
\bea\label{fhst}
\frac{d}{d s}\left(\frac{d}{d t}\left(f^s\left(h^t z\right)\right)\right)=J D^{2}H\left(f^s\left(h^t z\right)\right)\frac{d}{d s}\left(f^s\left(h^t z\right)\right).
\eea
Since $t$ and $s$ are independent, (\ref{fhst}) equals to
\bea\label{dtds}
\frac{d}{d t}\left(\frac{d}{d s}\left(f^s\left(h^t z\right)\right)\right)=J D^{2}H\left(f^s\left(h^t z\right)\right)\frac{d}{d s}\left(f^s\left(h^t z\right)\right),
\eea
hence at $s=0$ and $t=0$, we have
\bea
\frac{d}{d t}\left(J\nabla I\left( z\right)\right)=J D^{2}H\left(z\right)\left(J\nabla I\left(z\right)\right),\quad \text{for}\quad \forall z\in E,\nonumber
\eea
which means that $J\nabla I(z)$ is a solution of (\ref{eql4.1}).
\end{proof}
For the $2$-body problem, it's well known that it can be completely solved and the solution is Keplerian orbit. Consider the Hamiltonian function
of the $2$-body problem,
\bea
H_K(z)=\frac{1}{2}(p_1^2+p_2^2)-\frac{\hat{\lambda}}{(q_1^2+q_2^2)^{1/2}}\nonumber
\eea
where $z(t)=(p_{1}(t), p_{2}(t), q_{1}(t), q_{2}(t))^\mathcal{T}\in \mathbb{R}^4$ and $\hat{\lambda}$ is a non-zero constant.

Under the planar polar coordinates, the Hamiltonian equation of the $2$-body problem,
\bea\label{Kep1}
    z'(t)=J\nabla H_{K}(z), \ \
    z(0)=z(T).
\eea
has the well known Keplerian solution $z_{K}(t)=(p_{1}(t), p_{2}(t), q_{1}(t), q_{2}(t))^\mathcal{T}$
$$
p_{1}(t)=q'_{1}(t),\ \ p_{2}(t)=q'_{2}(t),\ \ q_{1}(t)=r(t)\cos\theta(t),\ \ q_{2}(t)=r(t)\sin\theta(t),
$$
where
$$
r(t)=\frac{C^2/\hat{\lambda}}{1+e\cos\theta(t)},\ \ r^2(\theta)\theta'=C,
$$
the constant $C$ is the angular momentum. Without loss of generality, we assume $C>0$. The linear Hamiltonian system at $z_{K}(t)$ is
\bea\label{Kepsolu1}
y'(t)=JD^{2}H_{K}(z_{K})y(t),
\eea
where
$$
D^2 H_{K}(z_{K})=\left(
            \begin{array}{cc}
              I_{2} & O_{2} \\
              O_{2} & \frac{\hat{\lambda}}{r^3(\theta)}(I_{2}-R(\theta)KR^{\mathcal{T}}(\theta)) \\
            \end{array}
          \right),
\ \
K=\left(
    \begin{array}{cc}
      3 & 0 \\
      0 & 0 \\
    \end{array}
  \right).
$$
Change the variable $t$ to  true anomaly
$\theta$ and perform linear symplectic transformation on $y(t)$,
\bea\label{tran1}
\tilde{y}(\theta)=\frac{1}{\sqrt[4]{\hat{\lambda}\mathfrak{p}}}\left(
                    \begin{array}{cc}
                      R^{-1}(\theta)r(\theta) & R^{-1}(\theta)r'(\theta) \\
                      O_{2} & \sqrt{\hat{\lambda}\mathfrak{p}}R^{-1}(\theta)r^{-1}(\theta) \\
                    \end{array}
                  \right)y(t(\theta)).
\eea
We obtain
\bea\label{Kepsolu2}
\dot{\tilde{y}}(\theta)=JB_{Kep}(\theta)\tilde{y}(\theta),
\eea
where
\bea
B(\theta)=B_{kep}(\theta)=\left( \begin{array}{cccc} I_{2} & -J_{2} \\
J_{2} & I_{2}-\frac{\mathcal{R}_{kep}}{1+e\cos\theta}
\end{array}\right),\ \ \mathcal{R}_{kep}=\left( \begin{array}{cccc} 3 & 0 \\
0 & 0
\end{array}\right)\nonumber
\eea
This tells us that if we get the fundamental solution of equation $(\ref{Kepsolu1})$, then by the transformation
(\ref{tran1}), we can obtain the fundamental solution of $(\ref{Kepsolu2})$, we write it as the following theorem.
\begin{thm}\label{FunKep}
    $\gamma_{Kep}(\theta)$ in (\ref{eq4.5}) is the fundamental solution of linear Hamiltonian system
    $$
    \dot{\gamma}(\theta)=JB_{Kep}(\theta)\gamma(\theta), \quad \gamma(0)=I_4.
    $$
\end{thm}
\begin{proof}
In the following, we first give the solution of $(\ref{Kepsolu1})$ by Lemma \ref{lemfm}
and then get the fundamental solution of (\ref{Kepsolu2}).
Since $B_{Kep}(e, \theta)$ doesn't depend on $C, \hat{\lambda}$, the fundamental solution of (\ref{Kepsolu2}) is also independent of these parameters. For this reason, without loss of generality, we take $C=\hat{\lambda}=1$ in the following computations, this makes the computation more clear.

It is well known that Keplerian solution has three first integrals, energy $H$,  angular momentum $C$ and Runge-Lenz vector $(A_1,A_2,0)$, with
\bea\label{H}
H=\frac{1}{2}(p_1^2+p_2^2)-\frac{1}{(q_1^2+q_2^2)^{1/2}},
\ \
C=p_2q_1-p_1q_2,\nonumber
\eea
\bea\label{A1}
A_1=\frac{q_1}{(q_1^2+q_2^2)^{1/2}}-p_2(p_2q_1-p_1q_2),
\ \
A_2=\frac{q_2}{(q_1^2+q_2^2)^{1/2}}+p_1(p_2q_1-p_1q_2).\nonumber
\eea
Based on Lemma \ref{lemfm}, we obtain
\bea
&&\xi_H(\theta(t))=J\nabla H(z_K)=(-\cos\theta(1+e\cos\theta)^2,\ -\sin\theta(1+\cos\theta)^2,\ -\sin\theta,\ e+\cos\theta)^\mathcal{T},\nonumber \\
&&\xi_C(\theta(t))=J\nabla C(z_K)=(-e-\cos\theta,-\sin\theta,\frac{-\sin\theta}{1+e\cos\theta},\frac{\cos\theta}{1+e\cos\theta})^\mathcal{T},\nonumber \\
&&\xi_{A_2}(\theta(t))=J\nabla A_2(z_K)
=(\sin2\theta+e\sin\theta(1+e\cos^2\theta),-\cos2\theta-e\cos^3\theta,\frac{1+e\cos\theta+\sin^2\theta}{1+e\cos\theta},\frac{\sin\theta\cos\theta}{1+e\cos\theta})^\mathcal{T}, \nonumber
\eea
are periodic solutions of (\ref{Kepsolu1}).




Notice that if $z_{K}(t)$ is the Keplerian solution of (\ref{Kep1}), then by scaling symmetry corresponding to invariant $HC^2$,
\bea
z_{K,h}(t)=(h^{\frac{1}{3}}p(ht),h^{-\frac{2}{3}}q(ht))^\mathcal{T} \nonumber
\eea
is also a solution of (\ref{Kep1}), with period $T_h=\frac{T}{h}$, $H(z_{K,h})=h^{\frac{2}{3}}H(z_{K,1})$,
from (\ref{dtds}), differentiating $z_{K,h}(t)$ with respect to $h$, we have that $\frac{d}{dh}z_{K,h}|_{h=1}(t)$ is a solution of (\ref{eql4.1}), but usually, it is not periodic.
Direct computation shows that
\bea
\xi_h(\theta(t))=\frac{d}{dh}z_{K,h}|_{h=1}(t)=\left(-\frac{\sin\theta}{3},\ \frac{e+\cos\theta}{3},\ -\frac{2\cos\theta}{3(1+e\cos\theta)},\ -\frac{2\sin\theta}{3(1+e\cos\theta)}\right)^\mathcal{T} + t\ \xi_H(\theta(t))
 \nonumber
\eea
is a non-periodic solution of (\ref{Kepsolu1}).

From above analysis, we can easily check that
\bea
\xi(\theta(t))=\left(\xi_H(\theta),\ \xi_C(\theta),\ \xi_{A_2}(\theta),\ \xi_h(\theta)\right) \nonumber
\eea
is non-degenerate, and then it is a fundamental solution of (\ref{Kepsolu1}).
Based on the transformation (\ref{tran1}), we obtain the fundamental solution of (\ref{Kepsolu2}) in the following
\bea
\begin{aligned}
\tilde{\gamma}_{Kep}(\theta)&=\mathbb{A}^-(\theta)\xi(\theta) \\
&=\left(\begin{array}{cccc}
    -1-e\cos\theta-e^2\sin^2\theta & -1 & \sin\theta-e\sin\phi\cos\theta &\frac{e\sin\theta}{1+e\cos\theta}-\rho_0(e,\theta)(1+e\cos\theta+e^2\sin^2\theta)  \\
    -e\sin\theta-e^2\sin\theta\cos\theta & 0 & e\sin^2\theta-\cos\theta & \frac{1}{3}-\rho_0(e,\theta)(e\sin\theta+e^2\sin\theta\cos\theta) \\
    e\sin\theta+e^2\cos\theta\sin\theta & 0 & \cos\theta+e\cos^2\theta & -\frac{2}{3}+\rho_0(e,\theta)(e\sin\theta+e^2\cos\theta\sin\theta) \\
    (1+e\cos\theta)^2 & 1 & -2\sin\theta-e\sin\theta\cos\theta & \rho_0(e,\theta)(1+e\cos\theta)^2
\end{array}\right),
\end{aligned} \nonumber
\eea
where $\rho_0(e,\theta)=\int_0^{\theta}\frac{d\tau}{(1+e\cos\tau)^2}$ and
\bea
    \mathbb{A}^-(\theta)=\left(
                    \begin{array}{cc}
                      R^{-1}(\theta)r(\theta) & R^{-1}(\theta)r'(\theta) \\
                      O_{2} & R^{-1}(\theta)r^{-1}(\theta) \\
                    \end{array}
                  \right)=
   \left( \begin{array}{cccc}
         \frac{\cos\theta}{1+e\cos\theta} & \frac{\sin\theta}{1+e\cos\theta} &-e\sin\theta\cos\theta & -e\sin^2\theta \\
         -\frac{\sin\theta}{1+e\cos\theta} & \frac{\cos\theta}{1+e\cos\theta} &-e\sin^2\theta & -e\sin\theta\cos\theta \\
         0 & 0 & (1+e\cos\theta)\cos\theta & (1+e\cos\theta)\sin\theta \\
         0 & 0 & -(1+e\cos\theta)\sin\theta & (1+e\cos\theta)\cos\theta
    \end{array}\right) \nonumber
\eea
Let
\bea\label{eq4.5}
\begin{aligned}
    \gamma_{Kep}(\theta)&=\tilde{\gamma}_{Kep}(\theta)\tilde{\gamma}_{Kep}^{-1}(0) \\
    &=\left(
\begin{array}{cccc}
    \frac{2+e-\cos\theta-e\sin^2\theta}{1+e} & \frac{2(-1+e\cos\theta)\sin\theta-\frac{3e(1+e)\sin\theta}{1+e\cos\theta}}{1-e} & -\frac{(1-e\cos\theta+\frac{3e}{1+e\cos\theta})\sin\theta}{1-e} & \frac{1-\cos\theta-e\sin^2\theta}{1+e} \\
    -\frac{(1+e\cos\theta)\sin\theta}{1+e} & -\frac{1+e-2\cos\theta+2e\sin^2\theta}{1-e} &  -\frac{1-\cos\theta+e\sin^2\theta}{1-e} & -\frac{(1+e\cos\theta)\sin\theta}{1+e} \\
    \frac{(1+e\cos\theta)\sin\theta}{1+e} & -\frac{2(-1+\cos\theta)(1+e+e\cos\theta)}{1-e} & -\frac{-2+\cos\theta+e\cos^2\theta}{1-e} & \frac{(1+e\cos\theta)\sin\theta}{1+e} \\
    -\frac{2(2+e+e\cos\theta)\sin^2(\frac{\theta}{2})}{1+e} & \frac{2(2+e\cos\theta)\sin\theta}{1-e} & \frac{(2+e\cos\theta)\sin\theta}{1-e} & \frac{-1+2\cos\theta+e\cos^2\theta}{1+e}
\end{array}
    \right)\\
    &\ \ +\rho_0(e,\theta)\left(\begin{array}{cccc}
        0 & \frac{3(1+e)(1+e\cos\theta+e^2\sin^2\theta)}{1-e} & \frac{3(1+e\cos\theta+e^2\sin^2\theta)}{1-e} & 0 \\
        0 & \frac{3e(1+e)(1+e\cos\theta)\sin\theta}{1-e} & \frac{3e(1+e\cos\theta)\sin\theta}{1-e} & 0 \\
         0 & -\frac{3e(1+e)(1+e\cos\theta)\sin\theta}{1-e} & -\frac{3e(1+e\cos\theta)\sin\theta}{1-e} & 0 \\
         0 & -\frac{3(1+e)(1+e\cos\theta)^2}{1-e} & -\frac{3(1+e\cos\theta)^2}{1-e} & 0
    \end{array}\right),
\end{aligned}
\eea
then $\gamma_{Kep}(\theta)$ is the fundamental solution of (\ref{Kepsolu2}) with $\gamma_{Kep}(0)=I_{4}$.
\end{proof}
\subsection{Quantitative analysis of the linear stability}
Based on the above observation, it's important to study the stability of following linear Hamiltonian system
\bea\label{1.9e2}
\dot{z}(\theta)=J(B_{Kep}(\theta)+\sigma D(\theta))z(\theta), \quad z(0)=-z(2\pi),
\eea
where $D(\theta)\in\mathcal{S}(2\mathfrak{n}), \sigma\in\mathbb{C}$. 
As in Section $2$, we consider operator $A|_{E_{\mathfrak{n}}(-I)}=-J_{2\mathfrak{n}}\frac{d}{dt}$ with domain
$$
E_{\mathfrak{n}}(-I)=\left\{z\in W^{1,2}([0,T];\mathbb{C}^{2\mathfrak{n}})\ |\ z(0)=-z(T) \right\}.
$$
where $B, D$ are bounded linear operators defined by $(Bz)(t) = B(t)z(t), (Dz)(t) = D(t)z(t)$ on $E_{\mathfrak{n}}(S)$. Then $A|_{E_{\mathfrak{n}}(-I)}$ is a
self-adjoint operator with compact resolvent, moreover for $\sigma\in \rho(A)$, the resolvent
set of $A|_{E_{\mathfrak{n}}(-I)}$, $(A|_{E_{\mathfrak{n}}(-I)}-\sigma I_{2\mathfrak{n}})^{-1}$ is Hilbert-Schmidt.

Now, we can prove Theorem \ref{thm1.1}, which give a quantitative analysis to the stable region of ERE in planar $N$-body problem.
Before proceeding with the proof, we need following Lemma in \cite{HS10},
\begin{lem}\label{lem4.2.1}
    Consider the fundamental solution $\gamma_{\sigma,e}(\theta)$ of (\ref{1.9e2}), then
$$
i_1(\gamma_{0,e})=0, \ \ \nu_{1}(\gamma_{0,e})=3,\ \ i_{\omega}(\gamma_{0,e})=2,\ \ \nu_{\omega}(\gamma_{0,e})=0, \ \ \text{for}\ \ \omega\neq1.
$$
\end{lem}
\begin{rem}
In \cite{HS10}, they first proved the index equalities in Lemma \ref{lem4.2.1} based on Gordon's lemma \cite{G77}, which says that the periodic elliptic Kepler orbits
are local minimizers of the action functional. We must mention that at that time, there was no precise expression of $\gamma_{0,e}(\theta)$ yet. In the present paper,
the first time, we give the precise expression of $\gamma_{0,e}(\theta)=\gamma_{Kep}(\theta)$ which is given in (\ref{eq4.5}).
\end{rem}
\begin{proof}[Proof of Theorem \ref{thm1.1}]
Under the assumption of Theorem \ref{thm1.1}, the eigenvalues of $D(A-B)^{-1}$ are real on $E_{\mathfrak{n}}(-I)$. Since $E_{\mathfrak{n}}(-I)=E_{\mathfrak{n}}^+(-I)\oplus E_{\mathfrak{n}}^-(-I) $
and $E_{\mathfrak{n}}^{\pm}(-I)$ are isomorphic to $\tilde{E}_{\mathfrak{n}}^\pm(-I)$, the eigenvalues of $D(A-B)^{-1}$ are also  real on $\tilde{E}_{\mathfrak{n}}(-I)$, it satisfies the condition in
Theorem \ref{thm3.3}. Applying Theorem \ref{thm3.3} for $E=\tilde{E}_{\mathfrak{n}}^{\pm}(-I)$, we only need to compute the precise expression such that
\bea\label{tra1}
Tr(\mathcal{F}^2(B_{Kep},\sigma D;\tilde{E}_{\mathfrak{n}}^\pm(-I)))<1.
\eea
Since
    \bea
Tr(\mathcal{F}^2(B_{Kep}, \sigma D; \tilde{E}_{\mathfrak{n}}^{\pm}(-I)))=\sigma^2 Tr(\mathcal{F}^2(B_{Kep}, D; \tilde{E}_{\mathfrak{n}}^{\pm}(-I))), \nonumber
    \eea
We denote
\bea\label{eq4.19}
f_\pm(e)=Tr(\mathcal{F}^2(B_{Kep},D;\tilde{E}_{\mathfrak{n}}^\pm(-I))),\ \ f(e)=\max\{f_+(e),\ f_-(e)\}.
\eea
Condition (\ref{tra1}) implies by
$$
|\sigma|<\frac{1}{\sqrt{f(e)}}.
$$
In the following, using the trace formula in Theorem \ref{th3.1}, we give a more precisely expression of $f_+(e),\ f_-(e)$.
From (\ref{2orTrace}) in Theorem \ref{th3.1}, we have
    \bea\label{comp1}
f_\pm(e)=Tr(\mathcal{F}^2(B_{Kep}, D; \tilde{E}_{\mathfrak{n}}^{\pm}(-I)))=Tr(G_{1,\pm}^2)-2 Tr(G_{2,\pm}).
    \eea
First we consider $E=\tilde{E}_{\mathfrak{n}}^{+}(-I)$, from (\ref{eq3.3})-(\ref{eq3.5}), we have
    \bea
\begin{aligned}
    Z_{0,+}=\left(
    \begin{array}{cccc}
        0 & 1 & 0 & 0 \\
        0 & 0 & 1 & 0
    \end{array}
    \right)^\mathcal{T},
\end{aligned}
\quad
\begin{aligned}
    Z_{1,+}=\left(
    \begin{array}{cccc}
        1 & 0 & 0 & 0 \\
        0 & 0 & 0 & 1
    \end{array}
    \right)^\mathcal{T},
\end{aligned}\nonumber
    \eea
therefore
\bea
P_+=\left(
\begin{array}{cccc}
    0 & 0 & \frac{3-e}{1-e} & \frac{2}{1-e} \\
    1 & 0 & 0 & 0 \\
    0 & 1 & 0 & 0 \\
    0 & 0 & -\frac{4}{1-e} & -\frac{3+e}{1-e}
\end{array}
\right), \quad
Q_{d,+}=\left(\begin{array}{cccc}
    0 & 0 & 0 & 0 \\
    1 & 0 & 0 & 0 \\
    0 & 1 & 0 & 0 \\
    0 & 0 & 0 & 0
\end{array}\right)
,\quad
Q_{d,+}P_+^{-1}=\left(\begin{array}{cccc}
    0 & 0 & 0 & 0 \\
    0 & 1 & 0 & 0 \\
    0 & 0 & 1 & 0 \\
    0 & 0 & 0 & 0
\end{array}\right).\nonumber
\eea
For simplicity of calculation, let
\bea
\mathcal{P}_+=\left(\begin{array}{cccc}
    0 & \frac{1}{1+e} & \frac{3(1+e)\pi}{2(1-e)(1-e^2)^{\frac{3}{2}}} & 0 \\
    1 & 0 & 0 & 0 \\
    -(1+e) & 0 & 0 & 1 \\
    0 & 0 & -\frac{3(1+e)^2\pi}{2(1-e)(1-e^2)^{\frac{3}{2}}} & 0
\end{array}\right),\ \
\text{and}\ \
\Gamma_+=\mathcal{P}_+^{-1}Q_{d,+}P_+^{-1}\mathcal{P}_+=\left(
\begin{array}{cccc}
    1 & 0 & 0 & 0 \\
    0 & 0 & 0 & 0 \\
    0 & 0 & 0 & 0 \\
    0 & 0 & 0 & 1
\end{array}
\right),\nonumber
\eea
we have
\bea\label{comp2}
\begin{aligned}
    Tr(G_{1,+}^2)&=Tr((P_+^{-1}M_1Q_{d,+})^2)
    =Tr((\tilde{M}_1^+\Gamma_+)^2),
\end{aligned}
\begin{aligned}
    Tr(G_{2,+})&=Tr(P_+^{-1}M_2Q_{d,+})
    =Tr(\tilde{M}_2^+\Gamma_+).
\end{aligned}
\eea
then from (\ref{comp1}) and (\ref{comp2}), we derive the computational formula for $f_{+}(e)$,
\bea
\begin{aligned}\label{eq4.14}
f_+(e)&=Tr((\tilde{M}_1^+\Gamma_+)^2)-2Tr(\tilde{M}_2^+\Gamma_+)
=-2\sum_{j=2}^3\int_0^{\pi}\int_0^{\theta}(\tilde{D}_{1j}^+(\theta)\tilde{D}_{j1}^+(s)+\tilde{D}_{4j}^+(\theta)\tilde{D}_{j4}^+(s))dsd\theta
\end{aligned}
\eea
where $\tilde{M}_k^+=\mathcal{P}_+^{-1}M_k\mathcal{P}_+$ and $\tilde{D}_{ij}^+=(\mathcal{P}_+^{-1}J\hat{D}(\theta)\mathcal{P}_+)_{ij}$.

Similar, for $E=\tilde{E}_{\mathfrak{n}}^-(-I)$, we have
\bea
Z_{0,-}=\left(
\begin{array}{cccc}
    1 & 0 & 0 & 0 \\
    0 & 0 & 0 & 1
\end{array}
\right)^\mathcal{T},\quad
Z_{1,-}=\left(
\begin{array}{cccc}
    0 & 1 & 0 & 0 \\
    0 & 0 & 1 & 0
\end{array}
\right)^\mathcal{T}\nonumber
\eea
and let
\bea
\mathcal{P}_-=\left(\begin{array}{cccc}
    1 & 0 & -\frac{3\pi(1-e)}{(1-e^2)^{\frac{3}{2}}} & -\frac{3\pi}{(1-e^2)^{\frac{3}{2}}} \\
    0 & 0 & -\frac{3-e}{1+e} & -\frac{2}{1+e} \\
    0 & 0 & \frac{4}{1+e} & \frac{3+e}{1+e} \\
    0 & 1 & \frac{3\pi}{\sqrt{1-e^2}} & \frac{3\pi(1+e)}{(1-e^2)^{\frac{3}{2}}}
\end{array}\right),
 \quad
\text{and}\ \
\Gamma_-=\mathcal{P}_-^{-1}Q_{d,-}P_-^{-1}\mathcal{P}_-=\left(
\begin{array}{cccc}
    0 & 0 & 0 & 0 \\
    0 & 0 & 0 & 0 \\
    0 & 0 & 1 & 0 \\
    0 & 0 & 0 & 1 \\
\end{array}
\right),\nonumber
\eea
then we have
\bea
\begin{aligned}\label{eq4.18}
f_-(e)
=Tr((\tilde{M_1}\Gamma_-)^2)-2Tr(\tilde{M_2}\Gamma_-)
=-2\sum_{j=2}^3\int_0^{\pi}\int_0^{\theta}(\tilde{D}^-_{3j}(\theta)\tilde{D}^-_{j3}(s)+\tilde{D}^-_{4j}(\theta)\tilde{D}^-_{j4}(s))dsd\theta
\end{aligned}
\eea
where $\tilde{M}_k^-=\mathcal{P}_-^{-1}M_k\mathcal{P}_-$ and $\tilde{D}_{ij}^-=(\mathcal{P}_-^{-1}J\hat{D}(\theta)\mathcal{P}_-)_{ij}$.
Then from Theorem \ref{thm3.3} and Lemma \ref{lem4.2.1}, we have
$$\mathcal{I}(A|_{E_{\mathfrak{n}}(-I)}-B_{0,e}, A|_{E_{\mathfrak{n}}(-I)}-B_{0,e}-\sigma D)=0,\ \ i_{-1}(\gamma_{\sigma,e})=i_{-1}(\gamma_{0,e})=2,\ \ \text{for}\ \ |\sigma|<\frac{1}{\sqrt{f(e)}},$$
where $f(e)=\max\{f_+(e),\ f_-(e)\}$.
Moreover, if for such $(\sigma,e)$, $i_1(\gamma_{\sigma,e})=0$ holds, then Theorem \ref{thm3.3} implies $\gamma_{\sigma,e}(T)$ is linearly stable and
$\gamma_{\sigma,e}(T)\approx R(\vartheta_{1})\diamond R(\vartheta_{2})$ with $\vartheta_{1},\vartheta_{2}\in(\pi,2\pi)$.
This completes the proof of the Theorem \ref{thm1.1}.
\end{proof}
We state a special case of Theorem \ref{thm1.1} as a Corollary, which can be easily obtained by using Corollary \ref{thm3.4}.
\begin{cor}\label{cor4.4}
    If $D>0$ or $D<0$, then for $(\sigma,e)$ satisfies
    \bea
|\sigma|<\frac{1}{\sqrt{f(e)}},\nonumber
    \eea
we have $\mathcal{I}(A|_{E_{\mathfrak{n}}(-I)}-B_{Kep},A|_{E_{\mathfrak{n}}(-I)}-B_{Kep}-\sigma D)=0$ and $i_{-1}(\gamma_{\sigma,e})=2$.
If for such $(\sigma,e)$, $i_1(\gamma_{\sigma,e})=0$ holds, then $\gamma_{\sigma,e}(T)$ is linearly stable and
$\gamma_{\sigma,e}(T)\approx R(\vartheta_{1})\diamond R(\vartheta_{2})$ with $\vartheta_{1},\vartheta_{2}\in(\pi,2\pi)$.
\end{cor}


\section{Applications}\label{sec4}
In order to study the linear stability of ERE via the index theory, we consider the Sturm-Liouville operator of the essential part (\ref{ga3e})-(\ref{hess}),
\bea \label{eqst}
\mathcal{A}=-\frac{d^2}{d\theta^2}I_k-2J_k\frac{d}{d\theta}+\frac{\mathcal{R}}{1+e\cos\theta}.\nonumber
\eea
$\mathcal{A}$ is a self-adjoint operator in $L^2([0,2\pi], \mathbb{C}^{k})$ with domain
$$ \mathbb{D}_{k}(\omega)=\{y\in W^{2,2}([0,T],\mathbb{C}^{k})\,|\,y(2\pi)=\omega y(0),\,\dot{y}(2\pi)=\omega \dot{y}(0)\}, \omega\in\mathbb{C}.$$
We define the $\omega$-Morse index $\phi_{\omega}(\mathcal{A})$ to be the total number of negative
eigenvalues of $\mathcal{A}$, and define
$\nu_{\omega}(\mathcal{A})=\dim\ker(\mathcal{A})$. Then we have the following theorem which relates the Morse index to the
Maslov-type index.

\begin{lem}\label{4.1} (See Long \cite{L02} p.172).
The $\omega$-Morse index $\phi_{\omega}(\mathcal{A})$ and nullity $\nu_{\omega}(\mathcal{A})$ are equal
to the $\omega$-Maslov-type index $i_{\omega}(\ga)$ and nullity $\nu_{\omega}(\ga)$ respectively, that is,
for any $\omega\in \mathbb{U}$, we have
\bea \phi_{\omega}(\mathcal{A})=i_{\omega}(\ga),\ \ \nu_{\omega}(\mathcal{A})=\nu_{\omega}(\ga). \label{indexequ}\nonumber\eea
where $\ga$ is given by (\ref{ga3e}).
\end{lem}
In our research of the linear stability of ERE, a useful form of the Sturm-Liouville operator $\mathcal{A}$ is
   \bea
\mathcal{A}=\mathcal{A}_++\kappa\frac{\tilde{N}}{1+e\cos\theta},\nonumber
   \eea
where $\mathcal{A}_+$ is a positive Sturm-Liouville operator and $\tilde{N}=\left(\begin{array}{cc}
       I & 0 \\
       0 & -I
   \end{array}\right)$, dimension of $I$ depends on $\mathcal{A}$, $\kappa\in\mathbb{C}$.
\begin{defi}\label{genEig}
In general, for two operators $\hat{A}, \hat{B}$ with same domain $E$, we call $\kappa$ is a generalized eigenvalue of $\hat{A}$, if
$$(\hat{A}+\kappa \hat{B})x=0$$ has non-trivial solution in $E$. Especially, when $\hat{A}$ is invertible, $\kappa$ is the generalized eigenvalue of $\hat{A}$ if and only if $\frac{1}{\kappa}$ is the usual eigenvalue of operator $-\hat{B}\hat{A}^{-1}$.
\end{defi}
\begin{prop}\label{prop2.1}
  The generalized eigenvalues of $\mathcal{A_{+}}$, or equivalently, the eigenvalues of $-\frac{\tilde{N}}{1+e\cos\theta}\hat{A}^{-1}$ are real.
\end{prop}
\begin{proof}
    Since $\mathcal{A}_+$ is positive, then $\mathcal{A}_+^{\frac{1}{2}}$ exists and it is also positive, $\mathcal{A}$ is similar to
$$
\mathcal{A}_+^{-\frac{1}{2}}\mathcal{A}\mathcal{A}_+^{-\frac{1}{2}}=I+\kappa\mathcal{A}_+^{-\frac{1}{2}}\frac{\tilde{N}}{1+e\cos\theta}\mathcal{A}_+^{-\frac{1}{2}}.
$$
Since $\mathcal{A}_+^{-\frac{1}{2}}\frac{\tilde{N}}{1+e\cos\theta}\mathcal{A}_+^{-\frac{1}{2}}$ is self-adjoint, thus the generalized eigenvalues of $\mathcal{A_{+}}$ are real..
\end{proof}
In the following, we will apply the stability criteria in Theorem \ref{thm1.1} to some classical EREs, and give a quantitative estimations of their stable regions or elliptic-hyperbolic regions over the full range $e\in[0,1]$. The details of computation are in Section $5$.
\subsection{Stable region of Lagrange solution}

In this subsection, let $\beta=\beta_{L}$. The essential part of linear Hamiltonian system at Lagrange solution is
\bea\label{eq5.1}
\dot{\gamma}_{\beta,e}(\theta)=J_4B_{L}(\theta)\gamma_{\beta,e}(\theta), \quad \gamma_{\beta,e}(0)=I_4
\eea
with
\bea\label{eq5.2}
B_{L}(\theta)=\left( \begin{array}{cccc} I_{2} & -J_{2} \\
J_{2} & I_{2}-\frac{\mathcal{R}_{L}}{1+e\cos\theta}
\end{array}\right),\ \ \mathcal{R}_{L}=\left(\begin{array}{cc}
    \frac{3+\sqrt{9-\beta_{L}}}{2} & 0 \\
    0 & \frac{3-\sqrt{9-\beta_{L}}}{2}
\end{array}\right),\ \ \beta_{L}\in[0,9],
\eea
we rewrite
\bea
B_{L}(\theta)=B_{Kep}(\theta)+\sigma_{L}(\beta)D_{L}(\theta), \nonumber
\eea
where
\bea\label{Dlag}
\sigma_{L}(\beta)=3-\sqrt{9-\beta},\ \ D_{L}(\theta)=\mathrm{diag}\{O_{2\times 2},\ \frac{\tilde{N}}{2(1+e\cos\theta)}\}, \ \ \tilde{N}=\left(
                                                                                                       \begin{array}{cc}
                                                                                                         1 & 0 \\
                                                                                                         0 & -1 \\
                                                                                                       \end{array}
                                                                                                     \right),\nonumber
\eea
For further calculations, we need the following lemma.
\begin{lem}\label{lem5.1}
   With the notations above, $\mathcal{F}(B_{Kep},D_{L},E_{2}(-I))$ and $\mathcal{F}(B_{Kep},D_{L},\tilde{E}_{2}^\pm(-I))$ have only real eigenvalues.
\end{lem}
\begin{proof}
Recall that
the Sturm-Liouville operator of (\ref{eq5.1})-(\ref{eq5.2}) is
\bea
\mathcal{A}_{L}(e,\beta,\theta)=-\frac{d^2}{d\theta^2}I_2-2J_2\frac{d}{d\theta}+\frac{\mathcal{R}_{L}}{1+e\cos\theta}, \nonumber
\eea
it is a self-adjoint operator with domain
$$ \mathbb{D}_{2}(\omega)=\{y\in W^{2,2}([0,T],\mathbb{C}^{2})\,|\,y(2\pi)=\omega y(0),\dot{y}(2\pi)=\omega \dot{y}(0)\}, \omega\in\mathbb{C}.$$
Denote by
    \bea
\mathcal{A}_{Kep}(e,\theta)=-\frac{d^2}{d\theta^2}I_2-2J_2\frac{d}{d\theta}+\frac{\mathcal{R}_{kep}}{1+e\cos\theta},\ \ \mathcal{R}_{kep}=\left( \begin{array}{cccc} 3 & 0 \\
0 & 0
\end{array}\right), \nonumber
    \eea
it's easy to see that $\mathcal{A}_{Kep}(e,\theta)=\mathcal{A}_{L}(e,0,\theta)$, it's the Sturm-Liouville operator corresponding to Keplerian solution.
Then we have
\bea\label{5.4}
\begin{aligned}
\mathcal{A}_{L}(e,9,\theta)=\mathcal{A}_{Kep}(e,\theta)-3\mathcal{R}_{1},
\end{aligned}
\eea
where
$$
\mathcal{A}_{L}(e,9,\theta)=-\frac{d^2}{d\theta^2}I_2-2J_2\frac{d}{d\theta}+\frac{3I_2}{2(1+e\cos\theta)},\ \ \mathcal{R}_{1}=\frac{\tilde{N}}{2(1+e\cos\theta)}.
$$
In \cite{HLS14}, they showed that $\mathcal{A}_{L}(e,9,\theta)$ is positive in domain $\mathbb{D}_{2}(\omega)$ for any $\omega\in\mathbb{C}$ and $e\in[0,1)$.
By Proposition \ref{prop2.1}, we know that the eigenvalues of
$\mathcal{R}_{1}\mathcal{A}_{L}(e,9,\theta)^{-1}$ are real, denote these eigenvalues by $\{\lambda_j\}_{j\in\mathbb{Z}}$. On the other hand,
from Lemma \ref{lem4.2.1} and Lemma \ref{4.1}, we know that $\mathcal{A}_{Kep}(e,\theta)$ is invertible in domain $\mathbb{D}_{2}(\omega)$ for any $\omega\neq1$. Then from (\ref{5.4}), it's easy to see that, the eigenvalues of
$\mathcal{R}_{1}\mathcal{A}_{Kep}(e,\theta)^{-1}$ are $\{\frac{\lambda_{i}}{1+3\lambda_{i}}\}\subset \mathbb{R}$. Combined with Remark \ref{detm},
$$\Pi\left(\mathcal{F}\left(B_{Kep},D_{L},E_{2}(-I)\right)\right)=\Pi\left(\mathcal{R}_1 \mathcal{A}_{Kep}^{-1}\right),\ \ \text{with}\ \ \mathcal{R}_{1}=\frac{\tilde{N}}{2(1+e\cos\theta)},$$
we obtain the eigenvalues of $\mathcal{F}(B_{Kep},D_{L},E_{2}(-I))$ are real. By the $\mathbb{Z}_{2}$ decomposition in Section $3.1$, the eigenvalues of $\mathcal{F}(B_{Kep},D_{L},\tilde{E}_{2}^\pm(-I))$ are real.
\end{proof}
In the following content, we always denote by
\bea\label{rho}
\begin{aligned}
&\rho_0(e,\theta)=\int_0^{\theta}\frac{1}{(1+e\cos s)^2}ds,\quad \rho_1(e)=\int_0^{\pi}\int_0^{\theta}\frac{1}{(1+e\cos s)^2}ds d\theta,\\
&\rho_2(e)=\int_0^{\pi}\int_0^{\theta}\frac{1}{1+e\cos s}ds d\theta,\quad \rho_3(e)=\int_0^{\pi}\int_0^{\theta}\rho_0(e,s)ds d\theta.
\end{aligned}
\eea

Then from Theorem \ref{thm1.1}, put $D(\theta)=D_{L}(e,\theta)$ into (\ref{eq4.14}) and (\ref{eq4.18}), we denote the functions $f(e), f_{\pm}(e)$ in this case as $f_{L,\pm}$ and $f_{L}(e)$, respectively.
With the help of software Mathematica, we have
\bea\label{eqflag+}
f_{L,+}(e)=a_{L,0}^+(e)+a_{L,1}^+(e)\rho_1(e)+a_{L,2}^+(e)\rho_2(e)+\int_0^{\pi}\int_0^{\theta}h_{L,+}(e,s,\theta)\rho_0^2(e,s)ds d\theta,
\eea
\bea\label{eqflag-}
f_{L,-}(e)=a_{L,0}^-(e)+a_{L,1}^-(e)\ \rho_1(e)+a_{L,2}^-(e)\ \rho_2(e)+a_{L,3}^-(e)\ \rho_3(e)+\int_0^{\pi}h_{L,-}(e,\theta)\ \rho_0^2(e,\theta)\  d\theta,
\eea
where the explicit expression for $a_{L,0}^\pm(e)$, $a_{L,1}^\pm(e)$, $a_{L,2}^\pm(e)$, $a_{L,3}^-(e)$ and $h_{L,\pm}(e,s,\theta)$ are given in Section \ref{sub6.1}.
By calculation we know that
$
f_{L,+}(e)<f_{L,-}(e),
$
hence
\bea\label{eqflag}
f_{L}(e)=\max\{f_{+}(e),f_{-}(e)\}=f_{L,-}(e).
\eea
Then from Theorem \ref{thm1.1} and Lemma \ref{lem5.1}, we have
$
i_{-1}(\gamma_{\beta,e})=2,
$ for
\bea\label{blag}
\sigma_{L}(\beta)<\frac{1}{\sqrt{f_{L}(e)}}\ \ \text{or equivalently,}\ \ \beta<9-(3-\frac{1}{\sqrt{f_{L}(e)}})^2.
\eea
%
From \cite{HLS14}, we know that
$i_1(\gamma_{\beta,e})=0$ for $\forall (\beta,e)\in [0,9)\times [0,1)$. Since $i_{-1}(\gamma_{\beta,e})=2$ for $\sigma_{L}(\beta)<1/\sqrt{f_{L}(e)}$, then
Theorem \ref{thm3.3} implies
$\gamma_{\beta,e}(2\pi)\approx R(\vartheta_1)\diamond R(\vartheta_2)$ for some $\vartheta_1, \vartheta_2 \in (\pi,2\pi)$. Thus $\gamma_{\beta,e}(2\pi)$ is linearly stable.
This give a quantitative analysis to the stable region of Lagrange solution in Theorem \ref{thm1.2}.

Formula (\ref{blag}) does give a precise curve of $(\beta,e)$ in its domain $[0,9)\times[0,1)$, but it contains integrals, which makes the formula seems so complicated. We will estimate these integrals in order to give a simpler formula. By doing that, we may lose some region, but the simpler formula is better both in aesthetic and application.
By estimating integrals in (\ref{blag}), we obtain two functions $g_{L,\pm}(e)$ satisfying
\bea \label{gf}
g_{L,+}(e)\geq f_{L,+}(e),\quad  g_{L,-}(e)\geq f_{L,-}(e),\ \ \forall e\in [0,1),
\eea
the detailed computation of the estimate and the explicit expression for $g_{L,\pm}(e)$ are given in Section \ref{sub6.2}. Based on this inequalities and Theorem \ref{thm1.2},
we have
\begin{cor}\label{bglag}
    Lagrange solution with mass parameter $\beta_{L}$ and eccentricity $e$ is linearly stable if
    \bea\label{eqbg-}
\beta_{L}<9-\left(3-1/\sqrt{g_{L,-}(e)}\right)^2,
    \eea
    where $g_{L,-}(e)$ is defined in (\ref{eqglag}).
\end{cor}

In \cite{HLS14}, they proved that there exist  two curves $\beta_s(e)$ and $\beta_m(e)$ satisfying $\beta_s\leq \beta_m(e)$ for $e\in [0,1)$ such that for $(\beta,e)$ satisfying $\beta<\beta_s(e)$,  Lagrange solution is strongly linearly stable; for $(\beta,e)$ satisfying $\beta\in(\beta_s(e),\beta_m(e))$, Lagrange solution is elliptic-hyperbolic; for $(\beta,e)$ satisfying $\beta\in(\beta_m(e),9]$, Lagrange solution is hyperbolic. These results give an analytical proof of the bifurcation diagram of \cite{MSS06}. Corollary \ref{bglag} gives an estimate of EE region quantitatively and (\ref{eqglag}) could give an explicit lower estimate of $\beta_s(e)$ and $\beta_m(e)$, see Figure \ref{fig:gLag}. The left curve in Figure \ref{fig:gLag} is given by  (\ref{eqbg-}) and the right one is by $$\beta_{L}<9-\left(3-1/\sqrt{g_{L,+}(e)}\right)^2,$$ the shadow area is a part of EE region.\ Two curves in Figure \ref{fig:gLag} intersect $\beta$ axis at about $(0,\ 0.7469)$, it is noticed that this bifurcation point is $(0,\ 0.75)$ precisely.
\begin{figure}[H]
    \centering
    \includegraphics[width=6cm]{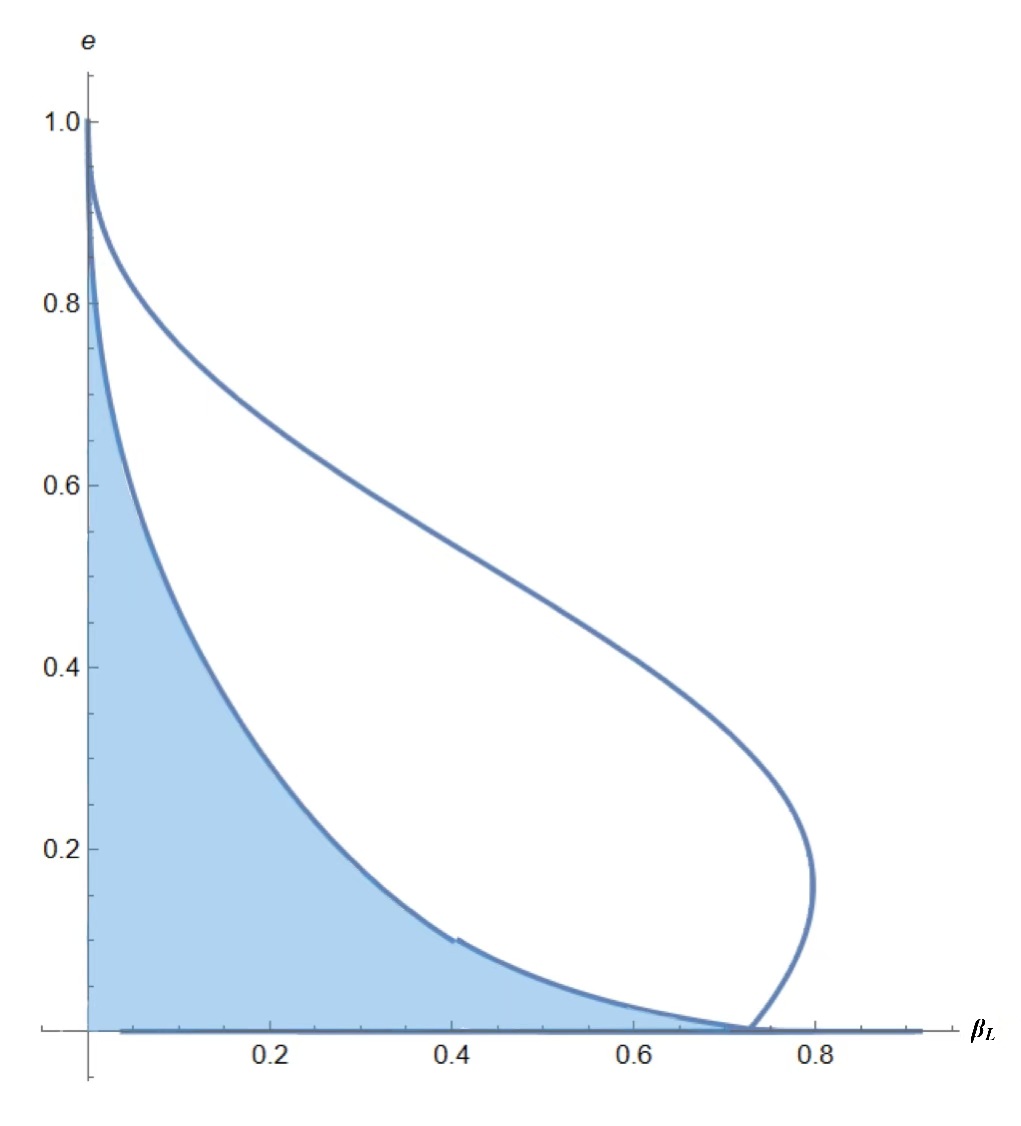}
    \caption{Curves $g_{L,-}(e)$ and $g_{L,+}(e)$.\ Two curves intersect $\beta$ axis at about $(0,\ 0.7469)$.\ The shadow is a lower estimate of EE region of Lagrange Solution.  Here we choose $e_0=0.1$, then $9-\left(3-1/\sqrt{\check{g}_{L,-}(0.1)}\right)^2=0.4077$ and $9-\left(3-1/\sqrt{\hat{g}_{L,-}(0.1)}\right)^2=0.4006$.}
    \label{fig:gLag}
\end{figure}
\begin{rem}
    Compared with the results in \cite{HOW15} and \cite{HOW19}, our estimates contain the whole range of eccentricity $e\in[0,1)$. Combine all of these regions, a better estimate could be obtained.
\end{rem}

\subsection{Elliptic-Hyperbolic region of Euler solution}

In this subsection, let $\beta=\beta_E$. The essential part of linearized Hamiltonian system at Euler solution is
\bea\label{eq5.13}
\dot{\gamma}_{\beta,e}(\theta)=J_4B_{E}(\theta)\gamma_{\beta,e}(\theta), \quad \gamma_{\beta,e}(0)=I_4,\nonumber
\eea
with
\bea\label{BE}
B_{E}(\theta)=\left( \begin{array}{cccc} I_{2} & -J_{2} \\
J_{2} & I_{2}-\frac{\mathcal{R}_{E}}{1+e\cos\theta}
\end{array}\right),\ \ \mathcal{R}_{E}=\left(\begin{array}{cc}
    2\beta_{E}+3 & 0 \\
    0 & -\beta_{E}
\end{array}\right),\ \ \beta_{E}\in[0,7],
\eea
then
\bea
B_E(\theta)=B_{Kep}(\theta)+Q_E(\theta),\ \
Q_E(\theta)=\mathrm{diag}\{O_{2\times2},\ \left(\begin{array}{cc}
    \frac{-2\beta}{1+e\cos\theta} & 0 \\
    0 & \frac{\beta}{1+e\cos\theta}
\end{array}\right)\}.\nonumber
\eea
Let
\bea
D_E(\theta)=\mathrm{diag}\{O_{2\times2},\ \frac{-\tilde{N}}{2(1+e\cos\theta)}\}=-D_{L}(e,\theta),\nonumber
\eea
then $Q_E(\theta)\leq 2\beta D_E(\theta)$. Thus we have
\bea
A-J-B_{Kep}(\theta)-2\beta D_E(\theta)\leq A-J-B_E(\theta)\ \ \text{in}\ \ \tilde{E}_{\mathfrak{n}}^\pm(-I),\nonumber
\eea
From \cite{ZL17}, we know that the relative Morse index $\mathcal{I}(A|_{E_{\mathfrak{n}}(-I)},A|_{E_{\mathfrak{n}}(-I)}-B_E)$ is non-decreasing, hence if $A-\nu J-B_{Kep}-2\beta D_E$ is non-degenerate for $\forall \beta\in[0, \varepsilon]$ with some $\varepsilon>0$, so is $A-\nu J-B_E$. We only need to estimate
the non-degeneracy of $A-J-B_{Kep}-2\beta D_E$ with $\beta=\beta_{E}$.
This closely resembles the case of Lagrange solution, where we estimate the non-degeneracy of operator $A-J-B_{Kep}-\sigma_{L}(\beta_{L}) D_{L}$ with $\sigma_{L}(\beta_{L})=3-\sqrt{9-\beta_{L}}$.
Similar to the case of Lagrange solution, applying Theorem \ref{thm1.1} to $A-J-B_{Kep}-2\beta_{E} D_E$ with domain $\tilde{E}_{\mathfrak{n}}^{\pm}(-I)$. The operator
$A-J-B_{Kep}-2\beta_{E} D_E$ is non-degenerate in $\tilde{E}_{\mathfrak{n}}^{\pm}(-I)$, for
$$
2\beta_{E}<\frac{1}{\sqrt{f_{E}(e)}},\ \ f_{E}(e)=\max\{f_{E,+}(e),f_{E,-}(e)\},
$$
where
\bea
f_{E,\pm}(e)=Tr(\mathcal{F}^2(B_{Kep}, D_E; \tilde{E}_{\mathfrak{n}}^{\pm}(-I)))\nonumber
\eea
Since $D_{E}(\theta)=-D_{L}(\theta)$, we have
\bea
f_{E,\pm}(e)=Tr(\mathcal{F}^2(B_{Kep}, D_E; \tilde{E}_{\mathfrak{n}}^{\pm}(-I)))=Tr(\mathcal{F}^2(B_{Kep}, D_{L}; \tilde{E}_{\mathfrak{n}}^{\pm}(-I)))=f_{L,\pm}(e), \nonumber
\eea
hence $f_{E}(e)=f_{L}(e)=\max\{f_{L,+}(e),f_{L,-}(e)\}$ and $A-J-B_E$ is non-degenerate in $\tilde{E}_{\mathfrak{n}}^{\pm}(-I)$ for
$$
\beta_{E}<\frac{1}{2\sqrt{f_{L}(e)}}.
$$
On the other hand, in Theorem 1.3 (\romannumeral3) and Theorem 1.5 (\romannumeral2) of \cite{ZL17}, they prove that for $(\beta_{E},e)\in [0,7] \times [0,1)$, there exist countable $\pm1$-degenerate curves from left to right in $[0,7]\times [0,1)$ by
\bea
0,\ \Theta_1^-,\ \Theta_1^+,\ \Gamma_1,\ \Theta_2^-,\ \Theta_2^+,\ \Gamma_2,\ \cdots,\ \Theta_n^-,\ \Theta_n^+,\ \Gamma_n,\ \cdots, \nonumber
\eea
where $0$ means the curve $\{(\beta_{E},e):\beta_{E}=0, e\in[0,1)\}$, by $\{\Theta_n^\pm\}_{n=1}^{\infty}$ means the $-1$-degenerate curves on $\tilde{E}_{-I}^{\pm}$, and by $\{\Gamma_n\}_{n=1}^{\infty}\cup\{0\}$ means the $1$-degenerate curves. Moreover, in the region between $0$ and $\Theta_1^-$, Euler solution is elliptic-hyperbolic. From above analysis, we know that the first
$-1$-degenerate curve $\Theta_1^-$ must on the right side of curve
$$
\{(1/2\sqrt{f_{L}(e)},e):e\in[0,1)\},
$$
hence the Euler solution is elliptic-hyperbolic for
$
\beta_{E}<1/(2\sqrt{f_{L}(e)}).
$
This give a quantitative analysis to the Elliptic-Hyperbolic region of Euler solution in Theorem \ref{thm1.2}.

Moreover, similar the Corollary \ref{bglag}, from inequality (\ref{gf}), we have
\begin{cor}\label{gE2}
    Euler solution with parameter $\beta_{E}$ and eccentricity $e$ is elliptic-hyperbolic if
    \bea\label{eqdE-}
\beta_{E}<\frac{1}{2\sqrt{g_{L,-}(e)}},
    \eea
where $g_{L,-}(e)$ are given by (\ref{eqglag}) in Section \ref{sub6.2}.
\end{cor}
 The region in Corollary \ref{gE2} is shown in Figure \ref{fig:gE-}, the curve is given by (\ref{eqdE-}), it intersects $\delta$ axis at about $(0,\ 0.0636)$.
\begin{figure}[H]
    \centering
    \includegraphics[width=6.2cm]{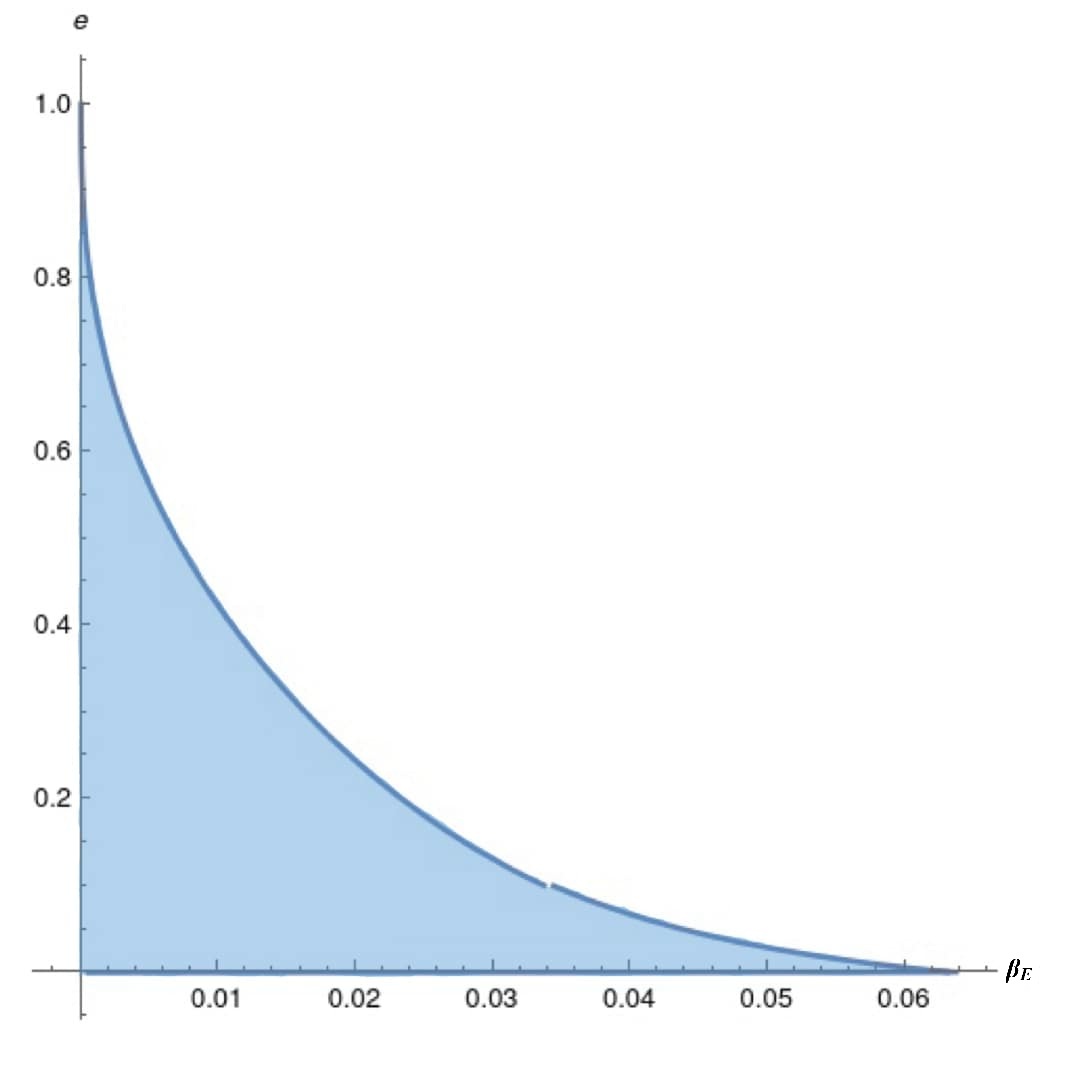}
    \caption{Estimate of EH region of Euler Solution. The curve intersects $\delta$ axis at about $(0,\ 0.0636)$. Here we choose $e_0=0.1$, then $1/(2\sqrt{\check{g}_{L,-}(0.1)})=0.0344$ and $1/(2\sqrt{\hat{g}_{L,-}(0.1)})=0.0338$. }
    \label{fig:gE-}
\end{figure}

\subsection{Stable region of the $(\alpha,\eta)$-type system}

In this subsection, we will introduce a more general case, the $(\alpha,\eta)$-type system which is useful in study the stability of system with high dimension. We have given a quantitative stability analysis for the planar $3$-body EREs, but as the number of celestial bodies increases, the dimensions of the system also increase, hence it's hard to analyze the stability. Based on the index theory, we can use a simplified system to control the original system, this simplified system can be decomposed to a series $(\alpha,\eta)$-type systems.
As an application, we will see how to use the $(\alpha,\eta)$-type system to estimate the stability of the regular $(1+n)$-gon system in Section $5.4$.

First, we consider
\bea\label{eq5.21}
\dot{\gamma}_{\zeta,e}(\theta)=J(B_{Kep}(\theta)+\zeta\tilde{D}(\theta))\gamma_{\zeta,e},\quad \gamma_{\zeta,e}(0)=I_4
\eea
with $\tilde{D}(e,\theta)=\mathrm{diag}\{O_{2\times 2},\ \frac{I_2}{1+e\cos\theta}\}$.
Put $D(\theta)=\tilde{D}(e,\theta)$ into (\ref{eq4.14}) and (\ref{eq4.18}), with the help of software Mathematica, we have
\bea\label{f-}
\tilde{f}(e)=\max\{f_{-}(e),f_{+}(e)\}=\tilde{a}_0(e)+\tilde{a}_1(e)\ \rho_1(e)+\tilde{a}_2(e)\ \rho_2(e)+\tilde{a}_3(e)\ \rho_3(e)+\int_0^{\pi}\tilde{h}(e,\theta)\ \rho_0^2(e,\theta)\ d\theta,
\eea
the explicit expression for $\tilde{a}_0(e)$, $\tilde{a}_1(e)$, $\tilde{a}_2(e)$, $\tilde{a}_3(e)$ and $\tilde{h}(e,\theta)$ are given in Section \ref{sub6.3}.
Then, by Corollary \ref{cor4.4}, we have
\begin{thm}\label{thm5.6}
    Suppose $\gamma_{\zeta,e}(\theta)$ satisfies (\ref{eq5.21}), then $i_{-1}(\gamma_{\zeta,e})=2$ when
    \bea
|\zeta|<\frac{1}{\sqrt{\tilde{f}(e)}},\nonumber
    \eea
    where $\tilde{f}(e)$ is defined in (\ref{f-}).
\end{thm}

Using Mathematica, we obtain the region of $\zeta$ in Theorem \ref{thm5.6} in Figure \ref{alpha.jpg}.
\begin{figure}[H]
    \centering
    \includegraphics[width=9cm]{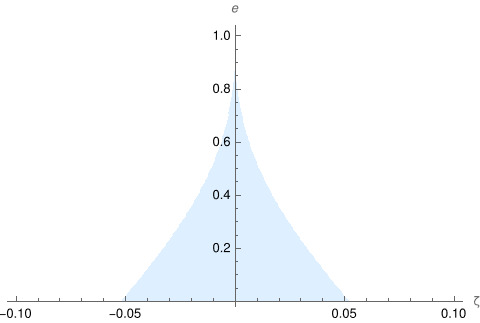}
    \caption{The curve intersects $\zeta$ axis at about $(0,\ -0.0523)$ and $(0,\ 0.0523)$. }
    \label{alpha.jpg}
\end{figure}

Similar the (\ref{gf}), we can estimating $\tilde{f}(e)$ to get a simpler formula, we have

\begin{cor}\label{gE}
    Suppose $\gamma_{\zeta,e}(\theta)$ satisfies (\ref{eq5.21}), then $i_{-1}(\gamma_{\zeta,e})=2$ when
    \bea
|\zeta|<\frac{1}{\sqrt{\tilde{g}(e)}},\nonumber
    \eea
    where
    \bea\label{tildg}
\tilde{f}(e)\leq\tilde{g}(e)=\left\{\begin{aligned}
    &\check{\tilde{g}}(e),\quad e\in[0,e_0),\\
    &\hat{\tilde{g}}(e),\quad e\in[e_0,1),
\end{aligned}\right.
\eea
with $e_0\in (0,\frac{224}{27\pi^2}]$, the explicit expression of $\check{\tilde{g}}(e), \hat{\tilde{g}}(e)$ are given by (\ref{tildg1}) and (\ref{tildg2}) in Section \ref{sub6.3}.
\end{cor}
Using Mathematica, we show the figure of Corollary \ref{gE} here.
\begin{figure}[H]
    \centering
    \includegraphics[width=4.7cm]{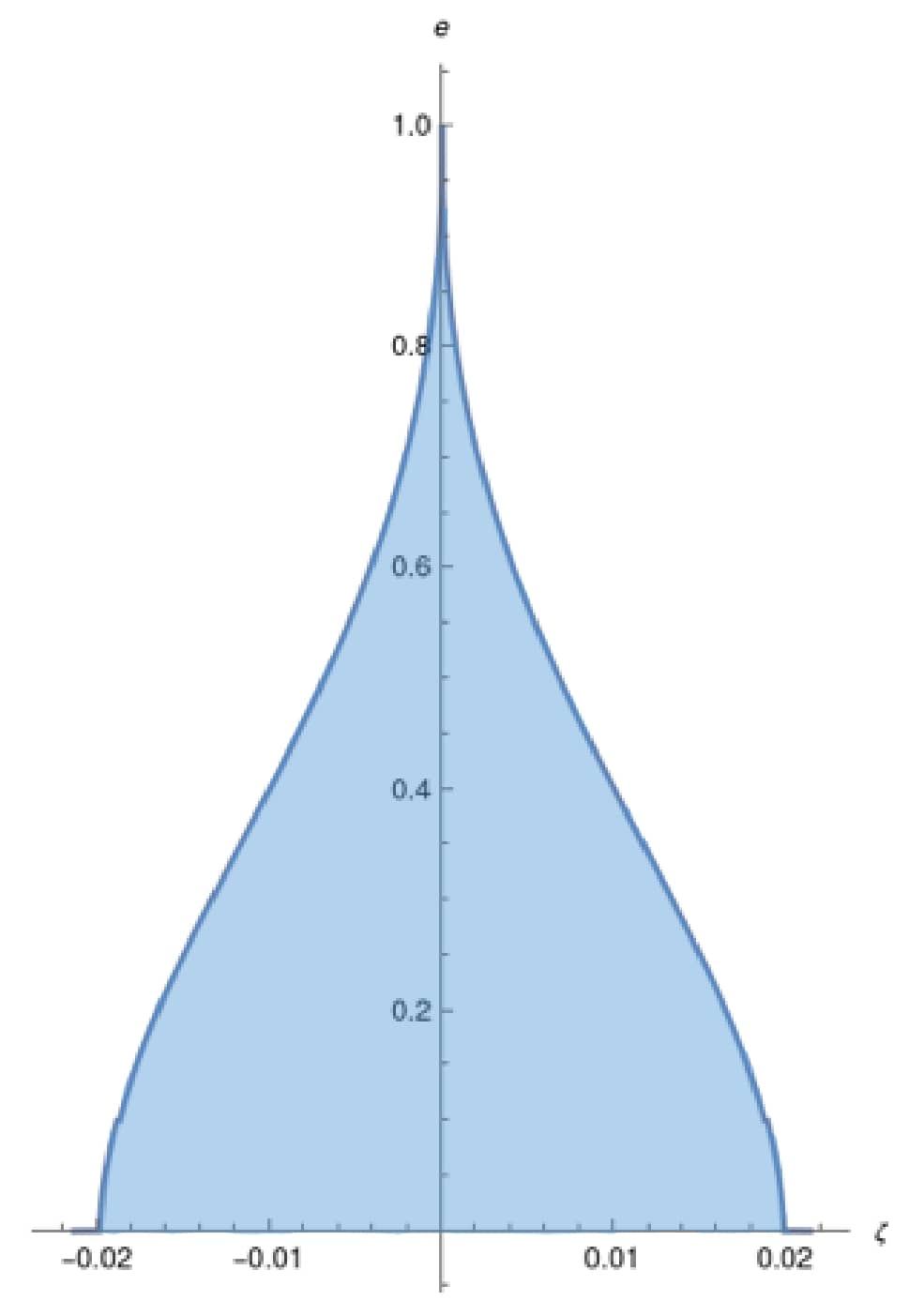}
    \caption{The curve intersects $\zeta$ axis at about $(0,\ -0.0193)$ and $(0,\ 0.0193)$. Here we choose $e_0=0.1$, then $1/(\check{\tilde{g}}_-(0.1))^{1/2}=0.0189,\ 1/(\hat{\tilde{g}}_-(0.1))^{1/2}=0.0187$. }
    \label{fig:tg-}
\end{figure}
Now, we can prove Theorem \ref{thm1.5},
\begin{proof}
We introduce the $(\alpha,\eta)$-type system, by
\bea\label{apl-eta}
B_{\alpha,\eta}(\theta)=\left(\begin{array}{cc}
    I_2 & -J_2 \\
    J_2 & I_2-\frac{R_{\alpha,\eta}}{1+e\cos\theta}
\end{array}\right),
\eea
where $R_{\alpha,\eta}=\alpha I_2+\eta \tilde{N}$ with $\alpha\geq 1$, $\eta\geq 0$ and $\tilde{N}=\left(\begin{array}{cc}
    1 & 0 \\
    0 & -1
\end{array}\right)$.
Let we consider the trace formula for
\bea\label{eq5.25}
\dot{\gamma}_{\alpha,\eta,e}(\theta)=JB_{\alpha,\eta}\gamma_{\alpha,\eta,e}(\theta),\quad \gamma_{\alpha,\eta,e}(0)=I_4,
\eea
Then
\bea
B_{\alpha,\eta}=B_{Kep}(\theta)+Q_{\alpha,\eta}(\theta),\nonumber
\eea
where
\bea
Q_{\alpha,\eta}(\theta)=\mathrm{diag}\{O_{2\times 2},\ \left(\begin{array}{cc}
    \frac{3-(\alpha+\eta)}{1+e\cos\theta} & 0  \\
    0 & \frac{\eta-\alpha}{1+e\cos\theta}
\end{array}\right)\}.\nonumber
\eea
For $\forall (\alpha,\eta)\in[1,\infty)\times [0,\infty)$, let $\zeta(\alpha,\eta)=\max\{|3-(\alpha+\eta)|,\ |\eta-\alpha|\}$, $\zeta$ is continuous with respect to $(\alpha,\eta)$, then $-\zeta(\alpha,\eta)\tilde{D}(\theta)\leq Q_{\alpha,\eta}(\theta)\leq\zeta(\alpha,\eta)\tilde{D}(\theta)$, hence
if $A-\nu J-B_{Kep}\pm\zeta(\alpha,\eta)\tilde{D}$ is non-degenerate for $\zeta(\alpha,\eta)\in[0,\varepsilon]$, so is $A-\nu J-B_{Kep}-Q_{\alpha,\eta}$. Then, from Theorem \ref{thm5.6}, we have
\bea
i_{-1}(\gamma_{\alpha,\eta,e})=2, \ \ \text{if}\ \ |\zeta(\alpha,\eta)|<\frac{1}{\sqrt{\tilde{f}(e)}},\nonumber
\eea
    where $\tilde{f}(e)$ is defined in (\ref{f-}). Moreover, if for such $(\alpha,\eta, e)$, $i_1(\gamma_{\alpha,\eta,e})=0$ holds, then Theorem \ref{thm3.3} implies $\gamma_{\alpha,\eta,e}(T)$ is linearly stable. This completes the proof of Theorem \ref{thm1.5}.
\end{proof}
From Theorem \ref{thm1.5} and Corollary \ref{gE}, we have
\begin{cor}\label{abg}
suppose $\gamma_{\alpha,\eta,e}(\theta)$ satisfies (\ref{eq5.25}), then $i_{-1}(\gamma_{\alpha,\eta,e})=2$ if
    \bea
|\zeta(\alpha,\eta)|<\frac{1}{\sqrt{\tilde{g}(e)}},\nonumber
    \eea
    where $\tilde{g}(e)$ is defined in (\ref{tildg}).
\end{cor}

\subsection{Stable region of regular $(1+n)$-gon solution}

Now we consider the regular $(1+n)$-gon solution. It is more complicated than Lagrange solution and Euler solution, since the essential part has higher dimension. But luckily, based on its beautiful symmetries, the essential part of the regular $(1+n)$-gon solution can be splitted on several invariant subspaces with lower dimension less than $8$, which is stated in the following lemma. Combined with the stability result of the $(\alpha,\eta)$- type system in Theorem \ref{thm1.5}, we can obtain the estimation of the stable
region of the regular $(1+n)$-gon solution.
\begin{lem}(Theorem 2.3 in \cite{HLO20})
    In the central configuration coordinates, the essential part of the linear Hamiltonian system
for the regular $(1+n)$-gon ERE is given by
\bea  \dot{\gamma}_{\beta,e}(\theta)=J_{4N-8}B_{\beta,e}(\theta)\gamma(\theta)\nonumber \eea
with $ N=1+n$ and
\bea
B(\theta)=\hat{B}_1(\theta) \diamond\cdots\diamond \hat{B}_{[\frac{n}{2}]}(\theta). \lb{bdcom}\nonumber \eea
\end{lem}
We list $\hat{B}_l(\theta)$ in the theorem above, and omit the sub-indices of $I$ and $\mathbb{J}$, which are chosen to have the same dimensions as those of $\mathcal{U}(l)$, readers can refer to \cite{HLO20} for the details. The expression of $\hat{B}_i(\theta), 1\leq i\leq{[\frac{n}{2}]}$ is given in (\ref{2.13}), it depends on mass parameter $\beta=\beta_{M}=1/m$, where $m$ is the mass of the body at the center.
\bea\label{2.13}
\hat{B}_{l}(\theta)=\left( \begin{array}{cccc} I & -\mathbb{J}\\
\mathbb{J} & I-r_e(\theta)R_l  \end{array}\right), \quad with \quad R_l=I+\frac{1}{\mu}\mathcal{U}(l), \, l=1,\cdots,[\frac{n}{2}],  \eea
and
\bea\label{2.17}
\mathcal{U}(1)=
\left(\begin{array}{cccc}\frac{n+m}{2}& 0 &\frac{3}{2}\sqrt{m(m+n)}& 0\\
0 & \frac{n+m}{2}& 0 & -\frac{3}{2}\sqrt{m(m+n)}\\ \frac{3}{2}\sqrt{m(m+n)} & 0 & \frac{m}{2}+2P_{1} & 0\\
0 & -\frac{3}{2}\sqrt{m(m+n)} & 0 & \frac{m}{2}+2P_{1}\end{array}\right),
\eea
\bea\label{gon-1}
\mathcal{U}(l)=\left(\begin{array}{cccc}a_{l} & 0 & 0 & S_{l}\\
0 & b_{l} & -S_{l} & 0\\ 0 & -S_{l} & a_{l} & 0\\
S_{l} & 0 & 0 & b_{l}\end{array}\right),\ \ \begin{array}{c} 2\leq l\leq [\frac{n-1}{2}]
\end{array},
\eea
\bea\label{gon-n}
\mathcal{U}(\frac{n}{2})=\left(\begin{array}{cc}
P_{\frac{n}{2}}-3Q_{\frac{n}{2}}+2m & 0 \\ 0 & P_{\frac{n}{2}}+3Q_{\frac{n}{2}}-m\end{array}\right),\ \
       \mathrm{if}\ \ n\in 2\mathbb{N},
\eea
where
\bea\label{n-gon co}
a_{l}=P_{l}-3Q_{l}+2m,\ \ b_{l}=P_{l}+3Q_{l}-m,\ \ \sigma_{n}=\frac{1}{2}\sum_{i=1}^{n-1}\csc \frac{\pi i}{n},\ \  \mu=\frac{1}{2}\sigma_{n}+m,
\eea
\bea
P_{l}=\sum_{j=1}^{n-1}\frac{1-\cos \theta_{jl}\cos \theta_{j}}{2d_{nj}^3},\ \
S_{l}=\sum_{j=1}^{n-1}\frac{\sin \theta_{jl}\sin \theta_{j}}{2d_{nj}^3},\ \
Q_{l}=\sum_{j=1}^{n-1}\frac{\cos \theta_{j}-\cos \theta_{jl}}{2d_{nj}^3},\nonumber
\eea
here $\theta_{jl}=\frac{2\pi jl}{n}$ and $d_{nj}=\|q_n-q_j\|$.

Based on the form of $R_l,\ l=1,\cdots,[\frac{n}{2}]$, in the regular $(1+n)$-gon central configuration, it can be estimated by $R_{\alpha,\eta}$.

For $l=1$,
\bea
R_{\check{\alpha}_{+}(m,n,1),\check{\eta}_{+}(m,n,1)}\oplus R_{\check{\alpha}_{-}(m,n,1),\check{\eta}_{-}(m,n,1)} \leq R_1\leq R_{\hat{\alpha}_{+}(m,n,1),\hat{\eta}_{+}(m,n,1)}\oplus R_{\hat{\alpha}_{-}(m,n,1),\hat{\eta}_{-}(m,n,1)}\nonumber
\eea
with
\bea
\check{\alpha}_{\pm}(m,n,1)=1+\frac{1}{\mu}(\check{d}_n+\frac{m}{2}),\quad \hat{\alpha}_{\pm}(m,n,1)=1+\frac{1}{\mu}(\hat{d}_n+\frac{m}{2}),\nonumber
\eea
\bea
\check{\eta}_+(m,n,1)=\hat{\eta}_+(m,n,1)=\frac{3\sqrt{m(m+n)}}{2\mu},\quad \check{\eta}_-(m,n,1)=\hat{\eta}_-(m,n,1)=-\frac{3\sqrt{m(m+n)}}{2\mu},\nonumber
\eea
and
\bea
\check{d}_n=\min\{2P_1,\frac{n}{2}\},\quad \hat{d}_n=\max\{2P_1,\frac{n}{2}\}.\nonumber
\eea

For $2\leq l\leq [\frac{n-1}{2}]$,
\bea
R_{\check{\alpha}_{+}(m,n,l),\check{\eta}_{+}(m,n,l)}\oplus R_{\check{\alpha}_{-}(m,n,l),\check{\eta}_{-}(m,n,l)} \leq R_l\leq R_{\hat{\alpha}_{+}(m,n,l),\hat{\eta}_{+}(m,n,l)}\oplus R_{\hat{\alpha}_{-}(m,n,l),\hat{\eta}_{-}(m,n,l)},\nonumber
\eea
with
\bea
\check{\alpha}_{\pm}(m,n,l)=1+\frac{1}{2\mu}(a_l+b_l-2S_l),\ \
\hat{\alpha}_{\pm}(m,n,l)=1+\frac{1}{2\mu}(a_l+b_l+2S_l),\nonumber
\eea
\bea
\check{\eta}_{\pm}(m,n,l)=\hat{\eta}_{\pm}(m,n,l)=\frac{1}{2\mu}(a_l-b_l). \nonumber
\eea

For $n\in2\mathbb{N}$, $l=[\frac{n}{2}]$,
\bea
R_{[\frac{n}{2}]}=R_{\alpha(m,n,[\frac{n}{2}]),\eta(m,n,[\frac{n}{2}])},\nonumber
\eea
with
\bea
\alpha(m,n,[\frac{n}{2}])=1+\frac{1}{2\mu}(a_{[\frac{n}{2}]}+b_{[\frac{n}{2}]}),\quad \eta(m,n,[\frac{n}{2}])=\frac{1}{2\mu}(a_{[\frac{n}{2}]}-b_{[\frac{n}{2}]}).\nonumber
\eea
The details of estimates can be found in \cite{HLO20}. By these useful estimates, they gave a region on which $i_1(\gamma_{\beta,e})=0$, we state a part of their results as the following lemma.
\begin{lem}\label{lem5.10}(See Theorem 4.8 in \cite{HLO20})
   If $n\geq 9$, then
   \bea
   \mathcal{A}(\mathcal{R},e)>0\ \ \text{in}\ \ \mathbb{D}_{n-1}(1),\ \ \text{and}\ \ i_1(\gamma_{\beta,e})=0\ \ \text{with}\ \ \beta=1/m,\ \
   \text{for}\ \ \forall (m,e)\in(2Q_{max}(n),\infty)\times [0,1),\nonumber
   \eea
   where $Q_{max}(n)=\max\{Q_l\ |\ 2\leq l\leq [\frac{n}{2}]\}$ and
   $$
   \mathcal{A}=-\frac{d^2}{d\theta^2}I_{2n-2}-2\mathbb{J}_{n-2}\frac{d}{d\theta}+\frac{\mathcal{R}}{1+e\cos\theta}, \ \ \mathcal{R}=R_{1}\oplus R_{2}\oplus\cdots\oplus R_{[\frac{n}{2}]}.
   $$
\end{lem}
But for $-1$-index, they proved existence only, we will calculate a precise region by trace formula, on which, the Sturm-Liouville operator for the regular $(1+n)$-gon with $n\geq 9$ is non-degenerate.

\begin{rem}
    For $n=8$ and $n=7$, in \cite{HLO20} and \cite{OS22} respectively, they proved that there exist some $m_0>0$ such that $\mathcal{A}>0$ on $\mathbb{D}_{n-1}(1)$ when $(m,e)\in(m_0(e),+\infty)\times [0,1)$, but by the technical restriction, they don't know the explicit $m_0(e)$ neither an estimation of $m_0(e)$.
\end{rem}
Now, in order to calculate a precise stable region, we need to use the stable results of the $(\alpha,\eta)$-type system in Corollary \ref{abg} for
\bea
\alpha=\check{\alpha}_{\pm}(m,n,l),\ \eta=\check{\eta}_{\pm}(m,n,l),\ \ \text{when}\ \ l=1,\cdots,[\frac{n-1}{2}], \nonumber
\eea
\bea
\alpha=\alpha(m,n,[\frac{n}{2}]),\ \eta=\eta(m,n,[\frac{n}{2}]),\ \text{when}\ \ n\in2\mathbb{N}, l=[\frac{n}{2}]. \nonumber
\eea
Then we obtain
\bea
\breve{\mathcal{ST}}_{n,l}^{\pm}=\{(m,e)\in[0,\infty)\times[0,1)\ |\ \zeta(\check{\alpha}_{\pm}(m,n,l),\check{\eta}_{\pm}(m,n,l))<\frac{1}{\sqrt{\tilde{g}(e)}}\},\nonumber
\eea
\bea
\hat{\mathcal{ST}}_{n,l}^{\pm}=\{(m,e)\in[0,\infty)\times[0,1)\ |\ \zeta(\hat{\alpha}_{\pm}(m,n,l),\hat{\eta}_{\pm}(m,n,l))<\frac{1}{\sqrt{\tilde{g}(e)}}\},\nonumber
\eea
\bea
\Tilde{\mathcal{ST}}_{n}=\left\{\ \begin{aligned}
    &\{(m,e)\in[0,\infty)\times[0,1)\ |\ \zeta(\alpha_{\pm}(m,n,[\frac{n}{2}]),\eta_{\pm}(m,n,[\frac{n}{2}]))<\frac{1}{\sqrt{\tilde{g}(e)}}\},\quad \text{if}\ \ n\in2\mathbb{N},\\
    &[0,\infty)\times [0,1)],\quad \text{if}\ \ n\in2\mathbb{N}+1.
\end{aligned}\right.\nonumber
\eea

Let $\mathcal{ST}_n=\tilde{\mathcal{ST}}_n\bigcap\limits_{l=1}^{[\frac{n-1}{2}]}(\hat{\mathcal{ST}}_{n,l}^+\bigcap\hat{\mathcal{ST}}_{n,l}^-\bigcap\check{\mathcal{ST}}_{n,l}^+\bigcap\check{\mathcal{ST}}_{n,l}^-)$, we have
\begin{thm}\label{thm5.12}
Let $\gamma_{\beta,e}$ be the foundmental solution of linearized Hamiltonian system of the regular $(1+n)$-gon central configuration with center mass $m$ and eccentricity $e$. For $\beta=1/m$, $(m,e)\in \mathcal{ST}_n$, we have
\bea
i_{-1}(\gamma_{\beta,e})=2N-4,\ \ \text{with}\ \ N=1+n.\nonumber
\eea
\end{thm}
\begin{proof}
    For $(m,e)\in\mathcal{ST}_n$, the foundamental solution for every $\mathcal{B}_l(\theta)$ has $-1$-Maslov-type index $2$, by symplectic addition of Maslov-type index, we
have
\bea
i_{-1}(\gamma_{\beta,e})=\left\{\ \begin{aligned}
    &4[\frac{n-1}{2}],\quad n=2\mathbb{N}+1,\\
    &4[\frac{n-1}{2}]+2,\quad n=2\mathbb{N}.
\end{aligned}\right. \nonumber
\eea
Since $N=1+n$, one can check that, we always have $i_{-1}(\gamma_{\beta,e})=2N-4$, the conclusion is proved.
\end{proof}

Denote by
\bea\label{EE1+n}
\mathcal{EE}_{(1+n)}=\mathcal{ST}_n\cap(2Q_{max}(n),\infty)\times [0,1)),\quad n\geq 9,
\eea
further we can prove Theorem \ref{thm1.4}.
\begin{proof}[Proof of Theorem \ref{thm1.4}]
    For $n\geq9$, by Lemma \ref{lem5.10} and Theorem \ref{thm5.12}, for $(m,e)\in\mathcal{EE}_{(1+n)}$, we have
    \bea
i_1(\gamma_{\beta,e})=0,\ \ i_{-1}(\gamma_{\beta,e})=2N-4.\nonumber
    \eea
Theorem \ref{thm3.3} implies $\gamma_{\beta,e}(2\pi)$ is linear stability.
\end{proof}

\begin{rem}
    $\mathcal{EE}_{(1+9)}$, a part of EE region of $(1+9)$-gon central configuration, is shown in Figure \ref{1+9.jpg}, the details of calculation are listed in Section 6.
    \begin{figure}[h]
    \centering
    \includegraphics[width=9cm]{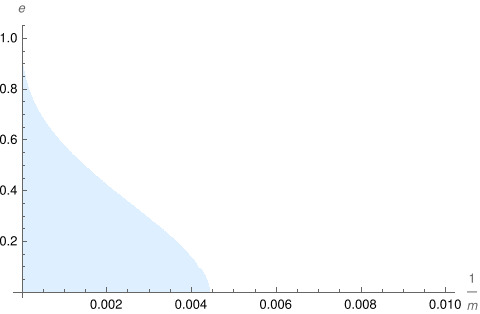}
    \caption{The curve intersects $\beta=1/m$ axis at about $(0,\ 0.00445)$. Here we choose $e_0=0.1$, then the two discontinuity points are $(0.00418,\ 0.1)$ and $(0.00422,\ 0.1)$. }
    \label{1+9.jpg}
\end{figure}
\end{rem}

\section{Appendix}\label{App}

\subsection{The explicit expression for $f_{L,\pm}(e)$}\label{sub6.1}
With the help of software Mathematica, we have
\bea
f_{L,+}(e)=a_{L,0}^+(e)+a_{L,1}^+(e)\rho_1(e)+a_{L,2}^+(e)\rho_2(e)+\int_0^{\pi}\int_0^{\theta}h_{L,+}(e,s,\theta)\rho_0^2(e,s)ds d\theta,\nonumber
\eea
and
\bea
f_{L,-}(e)=a_{L,0}^-(e)+a_{L,1}^-(e)\ \rho_1(e)+a_{L,2}^-(e)\ \rho_2(e)+a_{L,3}^-(e)\ \rho_3(e)+\int_0^{\pi}h_{L,-}(e,\theta)\ \rho_0^2(e,\theta)\  d\theta,\nonumber
\eea
where
\bea
\begin{aligned}
a_{L,0}^+(e)=&\frac{1}{36 e^4 \left(e^2-1\right)^3}\left(-382 e^7-54 \pi ^2 e^6+532 e^5+90 \pi ^2 e^4+120 e^3-54 \pi ^2 e^2+18 \pi ^2\right.\\
&\left.+18 \pi ^2 (e-1)^3 (e+1)^2 \left(e^2+1\right) \sqrt{(1+e)/(1-e)}-\left(45 e^8+249 e^6-300 e^4+6 e^2\right)\ln\left((1+e)/(1-e)\right)\right),
\end{aligned}\nonumber
\eea
\bea
\begin{aligned}
a_{L,0}^-(e)=&\frac{1}{36 e^4 \left(1-e^2\right)^5}\left(-382 e^{11}+72 \pi ^2 e^{10}+1296 e^9+81 \pi ^4 e^8-36 \pi ^2 e^8-1326 e^7+162 \pi ^4 e^6+558 \pi ^2 e^6\right.\\
&+292 e^5+81 \pi ^4 e^4-666 \pi ^2 e^4+120 e^3+90 \pi ^2 e^2+216 \pi ^2 \left(e^4-1\right) \sqrt{1-e^2} e^4 \log (e+1)\\
&+\pi ^2 \left(216 e^4-216 e^8\right) \sqrt{1-e^2} \log \left(\sqrt{1-e^2}/2+1/2\right)\\
&-3 \left(e^2-1\right) \left(15 e^8+68 e^6+6 \left(12 \pi ^2 \sqrt{1-e^2}+17\right) e^2+3 \left(24 \pi ^2 \sqrt{1-e^2}-61\right) e^4-2\right) e^2 \log \left((1+e)/(1-e)\right)\\
&\left.+\pi ^2 \left(126 e^{10}-108 e^9-180 e^8+648 e^7+486 e^6+756 e^5+954 e^4-108 e^2+18\right) \sqrt{1-e^2}-18 \pi ^2\right),\nonumber
\end{aligned}
\eea
\bea
a_{L,1}^+(e)=\frac{16-e^4}{4 \left(1-e^2\right)^2}, \quad a_{L,2}^+(e)=\frac{7 e^4+12 e^2-3}{2 e^2 \left(1-e^2\right)^2},\ \
a_{L,2}^-(e)=\frac{-7 e^4-12 e^2+3}{2 e^2 \left(1-e^2\right)^2},\quad a_{L,3}^-(e)=\frac{9 \pi  \left(e^2+1\right)^2}{\left(1-e^2\right)^{7/2}},\nonumber
\eea
\bea
a_{L,1}^-(e)=\frac{e^8-2 e^6-8 \left(9 \pi ^2 \sqrt{1-e^2}-4\right) e^2-4 \left(9 \pi ^2 \sqrt{1-e^2}+4\right)-3 \left(12 \pi ^2 \sqrt{1-e^2}+5\right) e^4}{4 \left(1-e^2\right)^4},\nonumber
\eea
\bea
h_{L,+}(e,s,\theta)=\frac{9 \left(e^3 \cos (3 s)+4 e^2 \cos (2 s)+\left(e^2+6\right) e \cos (s)+2 e^2+2\right) (e \cos (t)+1) \left(e^2 \cos (2 t)+2 e \cos (t)+1\right)}{4 \left(e^2-1\right)^2},\nonumber
\eea
\bea
\begin{aligned}
h_{L,-}(e,\theta)=
&\frac{1}{4 \left(e^2-1\right)^2}\left(3 \left(e^3 \cos (3 \theta)+4 e^2 \cos (2 \theta)+\left(e^2+6\right) e \cos (\theta)+2 e^2+2\right)\right)\\
&\cdot \left(3 \left(e^2+1\right) \theta+e \sin (\theta) \left(e^2 \cos (2 \theta)+2 e^2+6 e \cos (\theta)+9\right)\right).\nonumber
\end{aligned}
\eea
%
\subsection{The explicit expression for $g_{L,\pm}(e)$}\label{sub6.2}
In this subsection, we will show the detailed computation of $g_{L,\pm}(e)$.
We first give an estimation of (\ref{rho}), which is useful.
By calculation, we have:
\bea\label{est0}
    \frac{\theta}{(1-e^2)(1+e)}-\frac{e}{(1-e^2)(1-e)}\ \leq\  \rho_0(e,\theta)\ \leq\  \frac{\pi \theta \sqrt{(1-e)/(1+e)}}{\pi-e \theta}\ \leq\  \frac{\theta}{(1-e^2)^{\frac{3}{2}}},\nonumber
\eea
\bea\label{est1}
    \frac{\pi^2}{2(1-e)(e+1)^2}-\frac{\log \left((1+e)/(1-e)\right)}{1-e^2}\ \leq\  \rho_1(e)\ \leq\  \frac{\pi^2}{2(1-e^2)^{\frac{3}{2}}},\nonumber
\eea
\bea\label{est2}
    \frac{\pi^2}{2(1+e)}\ \leq\  \pi^2/2-2 e\ \leq\  \rho_2(e)\ \leq\  \left\{\begin{aligned}
        &\pi^2/2-2 e+\pi^2 e^2/4,\  e\in[0,e_0) \\
        &\frac{-\pi^2(e+\ln(1-e))}{e^2(1+e)},\ e\in[e_0,1)
    \end{aligned}\right. \leq \frac{-\pi^2(e+\ln(1-e))}{e^2(1+e)},\nonumber
\eea
\bea\label{est3}
    \rho_3(e)\ \leq\  \frac{\pi \log \left(\sqrt{1-e^2}/2+1/2\right)-\log (e+1)}{1-e^2}+\frac{\pi^3 \sqrt{(1-e)/(1+e)}((2-e) e+2 \times(1-e) \log (1-e))}{2 e^3\left(1-e^2\right)^{3 / 2}}.\nonumber
\eea
With these estimations, we obtain
$g_{L,+}(e)$ in (\ref{eqglag+}) by estimating (\ref{eqflag+}) and segmented function $g_{L,-}(e)$ in (\ref{eqglag}) by estimating (\ref{eqflag-}), which satisfy
\bea
f_{L,\pm}(e)\leq g_{L,\pm}(e). \nonumber
\eea
\bea\label{eqglag+}
\begin{aligned}
g_{L,+}(e)=&\frac{\left(5 e^4+3\right) \left(4 e-\pi ^2\right)}{4 e^2 \left(e^2-1\right)^2}-\frac{6 \pi ^2 \left(e^2+1\right) (e+\log(1-e) )}{(e-1)^2 e^2 (e+1)^3}-\frac{\pi ^2 \left(e^4-16\right)}{8 \left(1-e^2\right)^{7/2}}\\
&-\frac{3 \left(15 e^6+83 e^4-100 e^2+2\right) e^2 \log((1-e)/(1+e))}{36 e^4 \left(e^2-1\right)^3}\\
&+\frac{1}{48 \left(e^2-1\right)^5}\left[27 \left(3 \pi ^2-16\right) e^6+\left(-112+441 \pi ^2-18 \pi ^4\right) e^5-27 \left(48+3 \pi ^2+4 \pi ^4\right) e^4\right.\\
&\left.+\left(7536+972 \pi ^2-9 \pi ^4\right) e^3-90 \pi ^2 \left(6+\pi ^2\right) e^2+5184 e-18 \pi ^4\right]\\
&+\frac{1}{18 e^4 \left(e^2-1\right)^3}\left[-191 e^7-27 \pi ^2 e^6+266 e^5+45 \pi ^2 e^4+60 e^3-27 \pi ^2 e^2\right.\\
&\left.+9 \pi ^2 (e-1)^3 (e+1)^2 \left(e^2+1\right) \sqrt{(1-e)/(1+e)}+9 \pi ^2\right].
\end{aligned}
\eea
and
\bea\label{eqglag}
g_{Lag,-}(e)=\left\{\begin{aligned}
    &\check{g}_{Lag,-}(e),\quad e\in[0,e_0),\\
    &\hat{g}_{Lag,-}(e),\quad e\in[e_0,1).
\end{aligned}\right.\ \ \text{with}\ \ e_0\in (0,\frac{224}{27\pi^2}],
\eea
and
\bea \label{eqglag1}
\begin{aligned}
\check{g}_{L,-}(e)=&\frac{1}{72 e^4 \left(1-e^2\right)^5}\left[-45 \pi ^2 e^{14}+360 e^{13}+45 \pi ^2 e^{12}+225 \pi ^2 e^{11}-1844 e^{11}+27 \pi ^2 e^{10}-225 \pi ^2 e^9\right.\\
&+3888 e^9+162 \pi ^4 e^8-324 \pi ^3 e^8-45 \pi ^2 e^8-162 \pi ^4 e^7-648 \pi ^3 e^7-360 \pi ^2 e^7-3660 e^7-486 \pi ^4 e^6\\
&+1647 \pi ^2 e^6-324 \pi ^4 e^5+2592 \pi ^3 e^5+360 \pi ^2 e^5+1232 e^5+810 \pi ^4 e^4+4212 \pi ^3 e^4-1827 \pi ^2 e^4\\
&+810 \pi ^4 e^3+1944 \pi ^3 e^3+24 e^3+1458 \pi ^4 e^2+234 \pi ^2 e^2+972 \pi ^4 e-36 \pi ^2\\
&+216 \pi  \left(e^2+1\right) \left((2 \pi -3) e^2-2 \pi -3\right) \sqrt{1-e^2} e^4 \log (e+1)\\
&+\left(216 e^8+1296 e^6+1080 e^4\right) \left(\pi ^2 \sqrt{1-e^2} \log \left((1+e)/(1-e)\right)+\pi ^2 \sqrt{1-e^2} \log \left( \sqrt{1-e^2}/2+1/2\right)\right)\\
&-324 \pi ^3 (e-1) \left(2 e^6+8 e^5+3 \pi  e^4+16 e^4+6 \pi  e^3+16 e^3+6 \pi  e^2+6 e^2+6 \pi  e+3 \pi \right) \log (1-e)\\
&-648 \pi ^3 (e-1) \pi  \left(e^2+1\right)^2 \sqrt{(1+e)/(1-e)} e\log (1-e)\\
&+\left(-108 e^{12}-282 e^{10}+1776 e^8-2286 e^6+912 e^4-12 e^2\right) \log \left((1+e)/(1-e)\right)\\
&+\pi ^4 \left(-324 e^8-324 e^7-648 e^5+972 e^4-324 e^3+648 e^2\right) \sqrt{(1+e)/(1-e)}\\
&\left.+\pi ^2 \left(252 e^{10}-216 e^9-360 e^8+1296 e^7+972 e^6+1512 e^5+1908 e^4-216 e^2+36\right) \sqrt{1-e^2}\right],
\end{aligned}\nonumber
\eea
\bea \label{eqglag2}
\begin{aligned}
\hat{g}_{L,-}(e)=&\frac{1}{72 e^4 \left(1-e^2\right)^5}\left(225 \pi ^2 e^{11}-764 e^{11}+99 \pi ^2 e^{10}-405 \pi ^2 e^9+2592 e^9+162 \pi ^4 e^8-324 \pi ^3 e^8\right.\\
&-207 \pi ^2 e^8-162 \pi ^4 e^7-648 \pi ^3 e^7-2652 e^7-486 \pi ^4 e^6+1764 \pi ^2 e^6-324 \pi ^4 e^5+2592 \pi ^3 e^5+72 \pi ^2 e^5\\
&+584 e^5+810 \pi ^4 e^4+4212 \pi ^3 e^4-1908 \pi ^2 e^4+810 \pi ^4 e^3+1944 \pi ^3 e^3+216 \pi ^2 e^3+240 e^3+1458 \pi ^4 e^2\\
&+288 \pi ^2 e^2+972 \pi ^4 e-108 \pi ^2 e-36 \pi ^2\\
&-648 \pi ^4 (e-1) \sqrt{(1-e)/(e+1)} \left(e^2+1\right)^2 e \log(1-e)\\
&-36 \pi ^2 (e-1) \left(-5 e^8+18 \pi  e^6+10 e^6+72 \pi  e^5+27 \pi ^2 e^4+144 \pi  e^4-8 e^4+54 \pi ^2 e^3+144 \pi  e^3+54 \pi ^2 e^2\right.\\
&\left.+54 \pi  e^2+6 e^2+54 \pi ^2 e+27 \pi ^2-3\right)\log(1-e)\\
&+216 \pi  \left((2 \pi -3) e^2-2 \pi -3\right) \left(e^2+1\right) \sqrt{1-e^2} e^4 \log(e+1) \\
&+\left(216 e^8+1296 e^6+1080 e^4\right)\pi ^2 \sqrt{1-e^2}\left(\log((1+e)/(1-e))+\log(\sqrt{1-e^2}/2+1/2)\right)\\
&+\left(-108 e^{12}-282 e^{10}+1776 e^8-2286 e^6+912 e^4-12 e^2\right)\log((1+e)/(1-e))\\
&+\pi ^4 \left(-324 e^8-324 e^7-648 e^5+972 e^4-324 e^3+648 e^2\right) \sqrt{(1+e)/(1-e)}\\
&\left.+\pi ^2 \left(252 e^{10}-216 e^9-360 e^8+1296 e^7+972 e^6+1512 e^5+1908 e^4-216 e^2+36\right) \sqrt{1-e^2}\right].
\end{aligned}\nonumber
\eea

\subsection{The explicit expression for $\tilde{f}(e)$ and $\tilde{g}(e)$}\label{sub6.3}
With the help of software Mathematica, we have
\bea
\tilde{f}(e)=\tilde{a}_0(e)+\tilde{a}_1(e)\ \rho_1(e)+\tilde{a}_2(e)\ \rho_2(e)+\tilde{a}_3(e)\ \rho_3(e)+\int_0^{\pi}\tilde{h}(e,\theta)\ \rho_0^2(e,\theta)\ d\theta,\nonumber
\eea
where
\bea
\begin{aligned}
\tilde{a}_0(e)=&\frac{1}{e^4 \left(1-e^2\right)^5}\left[140 e^{11}+29 \pi ^2 e^{10}+6 e^9+36 \pi ^4 e^8+17 \pi ^2 e^8-408 e^7+36 \pi ^4 e^6+47 \pi ^2 e^6+238 e^5\right.\\
&+9 \pi ^4 e^4-101 \pi ^2 e^4+24 e^3+10 \pi ^2 e^2-2 \pi ^2-36 \pi ^2 \sqrt{1-e^2} \left(-2 e^4+e^2+1\right) e^4 \log(e+1)\\
&+\pi ^2 \left(-72 e^8+36 e^6+36 e^4\right) \sqrt{1-e^2}\log\left(\sqrt{1-e^2}/2+1/2\right)\\
&+\left(e^2-1\right) \left(6 e^8-47 e^6-3 \left(12 \pi ^2 \sqrt{1-e^2}+13\right) e^2+\left(78-72 \pi ^2 \sqrt{1-e^2}\right) e^4+2\right) e^2\log\left((1+e)/(1-e)\right)\\
&\left.+\pi ^2 \left(-26 e^{10}+72 e^9+20 e^8+180 e^7+172 e^6+72 e^5+170 e^4-14 e^2+2\right) \sqrt{1-e^2}\right],
\end{aligned}\nonumber
\eea
\bea
\tilde{a}_1(e)=-\frac{3 \left(4 e^4+43 e^2+28\right)}{\left(1-e^2\right)^2}-\frac{36 \left(2 \pi  e^2+\pi \right)^2}{\left(1-e^2\right)^{7/2}},\ \
\tilde{a}_2(e)=\frac{2 \left(18 e^4+e^2+5\right)}{e^2 \left(1-e^2\right)^2},\quad  \tilde{a}_3(e)=\frac{36 \pi  \left(2 e^2+1\right)^2}{\left(1-e^2\right)^{7/2}},\nonumber
\eea
\bea
\tilde{h}(e,\theta)=\frac{18 (e \cos (\theta )+1) \left(e^2+2 e \cos (\theta )+1\right) \left(2 e^2 \theta +e \sin (\theta ) \left(e^2+e \cos (\theta )+3\right)+\theta \right)}{\left(1-e^2\right)^2}.\nonumber
\eea
and
 \bea
\tilde{f}(e)\leq\tilde{g}(e)=\left\{\begin{aligned}
    &\check{\tilde{g}}(e),\quad e\in[0,e_0),\\
    &\hat{\tilde{g}}(e),\quad e\in[e_0,1),
\end{aligned}\right.\nonumber
\eea
\bea \label{tildg1}
\begin{aligned}
\check{\tilde{g}}(e)&=\frac{1}{e^4(1-e^2)^5}\left[\sqrt{1-e^2} \left(-\pi ^2 \left(-72 e^8+36 e^6+36 e^4\right) \log (1-e)\right.\right.\\
&-\pi  \left(144 e^8+144 e^6+36 e^4\right) \log (e+1)
+\pi ^2 \left(72 e^8+180 e^6+72 e^4\right) \log \left(\sqrt{1-e^2}/2+1/2\right)\\
&\left.+\pi ^2 \left(-26 e^{10}+72 e^9+20 e^8+180 e^7+172 e^6+72 e^5+170 e^4-14 e^2+2\right)\right)\\
&-9 \pi ^2 e^{14}+72 e^{13}+17 \pi ^2 e^{12}/2-72 e^{11}+99 \pi ^2 e^{10}/2-72 \pi ^2 e^9+534 e^9+54 \pi ^4 e^8-273 \pi ^2 e^8/8\\
&-252 \pi ^2 e^7+528 e^7+54 \pi ^4 e^6+77 \pi ^2 e^6/2-108 \pi ^2 e^5+726 e^5+27 \pi ^4 e^4/2-225 \pi ^2 e^4/2+4 e^3+15 \pi ^2 e^2\\
&+\left(e^2-1\right)^3 \left(6 e^4-35 e^2+2\right) e^2 \log (e+1)+\pi ^4 \left(-72 e^7+144 e^6-72 e^5+144 e^4-18 e^3+36 e^2\right) \sqrt{(1-e)/(e+1)}\\
&-(e-1) \left(6 e^{10}+6 e^9-47 e^8-47 e^7+78 e^6+78 e^5+3 \left(48 \pi ^4 \sqrt{(1-e)/(e+1)}-13\right) e^4-39 e^3\right.\\
&\left.\left.+2 \left(72 \pi ^4 \sqrt{(1-e)/(e+1)}+1\right) e^2+2 e+36 \pi ^4 \sqrt{(1-e)/(e+1)}\right) e \log (1-e)-2 \pi ^2\right],
\end{aligned}
\eea

\bea \label{tildg2}
\begin{aligned}
\hat{\tilde{g}}(e)&=\frac{1}{e^4(1-e^2)^5\sqrt{1-e^2}}\left[\sqrt{1-e^2} \left(140 e^{11}+121 \pi ^2 e^{10}/2-108 \pi ^2 e^9+310 e^9+54 \pi ^4 e^8-505 \pi ^2 e^8/8\right.\right.\\
&-182 \pi ^2 e^7+648 e^7+54 \pi ^4 e^6+115 \pi ^2 e^6/2-150 \pi ^2 e^5+670 e^5+27 \pi ^4 e^4/2-119 \pi ^2 e^4\\
&-10 \pi ^2e-2\pi^2+18 \pi ^2 e^3+24 e^3+20 \pi ^2 e^2+\left(e^2-1\right)^3 \left(6 e^4-35 e^2+2\right) e^2 \log (e+1)\\
&\left.-(e-1)^3 (e+1)^2 \left(6 e^7+6 e^6-35 e^5-\left(35+36 \pi ^2\right) e^4+2 e^3-2 \left(\pi ^2-1\right) e^2-10 \pi ^2\right) \log (1-e)\right)\\
&+2 \pi  \left(18 \left(4 e^6-3 e^2-1\right) e^4 \log (e+1)-18 \pi  (e-1)^2 \left(2 e^2+1\right) \left(e^5+2 e^4+e^3-2 \pi ^2 e^2-\pi ^2\right) e \log (1-e)\right.\\
&+\pi  \left(13 e^{12}-36 e^{11}-23 e^{10}-54 e^9+4 \left(9 \pi ^2-19\right) e^8-54 \left(2 \pi ^2-1\right) e^7+\left(1+108 \pi ^2\right) e^6\right.\\
&\left.\left.\left.-36 \left(3 \pi ^2-1\right) e^5+\left(92+81 \pi ^2\right) e^4-27 \pi ^2 e^3+2 \left(9 \pi ^2-4\right) e^2\right.\right.\right.\\
&\left.\left.\left.-18 \left(2 e^6+3 e^4-3 e^2-2\right) e^4 \log \left(\sqrt{1-e^2}/2+1/2\right)+1\right)\right)\right].
\end{aligned}
\eea

\subsection{Computation of region $\mathcal{EE}_{(1+9)}$}
For $n=9$, we have
\bea
\begin{aligned}
\check{\alpha}_{\pm}(m,9,1)=&\left[18 m+3/128\csc^3(\pi/9)\sec^3(\pi/9)\sec^3(\pi/18)\left(21\sqrt{3}+5\sin(\pi/9)+23\sin(2\pi/9)+3\sqrt{3}\sin(\pi/8)\right.\right.\\
&\left.\left.+55\cos(\pi/18)+6\sqrt{3}\cos(2\pi/9)\right)\right]\big{/}\left(12 m+4\sqrt{3}+6\csc(\pi/9)+6\csc(12\pi/9)+6\csc(\pi/18)\right),
\end{aligned}\nonumber
\eea
\bea
\begin{aligned}
    \check{\alpha}_{\pm}(m,9,2)=&9 \csc ^3\left(\pi/9\right) \sec ^3\left(\pi/18\right) \sec ^3\left(\pi/9\right)
    \left(3 m\left(1+\sin \left(\pi/18\right)+2 \cos \left(2 \pi/9\right)\right)+4 \sin \left(2 \pi/9\right)
    +\sqrt{3} \sin \left(\pi/18\right)\right.\\
    &\left.+5 \sqrt{3}+13 \cos \left(\pi/18\right)+2 \sqrt{3} \cos \left(2 \pi/9\right)\right)\big{/}\left(256(6 m+2 \sqrt{3}+3 \csc (\pi/9)+3 \csc (2 \pi/9)+3 \sec (\pi/18))\right),
    \end{aligned}\nonumber
    \eea
\bea
\hat{\alpha}_{\pm}(m,9,1)=1+3(m+9)/(6 m+2\sqrt{3}+3\csc(\pi/9)+3\csc(2\pi/9)+3\sec(\pi/18)),\nonumber
\eea
    \bea
    \begin{aligned}
        \hat{\alpha}_{\pm}(m,9,2)=&\frac{1}{256\left(6 m+2 \sqrt{3}+3 \csc \left(\pi/9\right)+3 \csc \left(2 \pi/9\right)+3 \sec \left(\pi/18\right)\right)}\left[3 \csc ^3\left(\pi/9\right) \sec ^3\left(\pi/18\right) \sec ^3\left(\pi/9\right)\right.\\
        &\left(9 m\left(1+\sin \left(\pi/18\right)+2 \cos \left(2 \pi/9\right)\right)+25 \sqrt{3}+12 \sin \left(\pi/9\right)+38 \sin \left(2 \pi/9\right)+2 \sqrt{3} \sin \left(\pi/18\right)\right.\\
        &+68 \cos \left(\pi/18\right)\left.\left.+4 \sqrt{3} \cos \left(2 \pi/9\right)\right)\right],
        \end{aligned}\nonumber
        \eea
\bea
\begin{aligned}
    \check{\alpha}_{\pm}(m,9,3)=&\frac{1}{512\left(6 m+2 \sqrt{3}+3 \csc \left(\pi/9\right)+3 \csc \left(2\pi/9\right)+3 \sec \left(\pi/18\right)\right)}\left[3 \csc ^3\left(\pi/9\right) \sec ^3\left(\pi/18\right) \sec ^3\left(\pi/9\right)\right.\\
    &\left(18 m\left(1+\sin \left(\pi/18\right)+2 \cos \left(2 \pi/9\right)\right)+42 \sqrt{3}+18 \sin \left(\pi/9\right)+66 \sin \left(2 \pi/9\right)+\sqrt{3} \sin \left(\pi/18\right)\right.\\
    &\left.\left.+129 \cos \left(\pi/18\right)-12 \sqrt{3} \cos \left(\pi/9\right)+2 \sqrt{3} \cos \left(2 \pi/9\right)\right)\right],
    \end{aligned}\nonumber
    \eea
\bea
\begin{aligned}
        \hat{\alpha}_{\pm}(m,9,3)=&\frac{1}{512\left(6 m+2 \sqrt{3}+3 \csc \left(\pi/9\right)+3 \csc \left(2 \pi/9\right)+3 \sec \left(\pi/18\right)\right)}\left[3 \csc ^3\left(\pi/9\right) \sec ^3\left(\pi/18\right) \sec ^3\left(\pi/9\right)\right.\\
        &\left(18 m\left(1+\sin \left(\pi/18\right)+2 \cos \left(2 \pi/9\right)\right)+54 \sqrt{3}+18 \sin \left(\pi/9\right)+66 \sin \left(2 \pi/9\right)+11 \sqrt{3} \sin \left(\pi/18\right)\right.\\
        &\left.\left.+129 \cos \left(\pi/18\right)+12 \sqrt{3} \cos \left(\pi/9\right)+22 \sqrt{3} \cos \left(2 \pi/9\right)\right)\right],
        \end{aligned}\nonumber
\eea
\bea
\begin{aligned}
    \check{\alpha}_{\pm}(m,9,4)=&\frac{1}{2048\left(6 m+2 \sqrt{3}+3 \csc \left(\pi/9\right)+3 \csc \left(2 \pi/9\right)+3 \sec \left(\pi/18\right)\right)}\left[3 \csc ^3\left(\pi/18\right) \sec ^6\left(\pi/18\right) \sec ^3\left(\pi/9\right)\right.\\
    &\left(9 m\left(1+\sin \left(\pi/18\right)+2 \cos \left(2 \pi/9\right)\right)+25 \sqrt{3}+12 \sin \left(\pi/9\right)+38 \sin \left(2 \pi/9\right)+2 \sqrt{3} \sin \left(\pi/18\right)\right.\\
    &\left.\left.+68 \cos \left(\pi/18\right)+4 \sqrt{3} \cos \left(2 \pi/9\right)\right)\right],
    \end{aligned}\nonumber
\eea
\bea
\begin{aligned}
    \hat{\alpha}_{\pm}(m,9,4)=&\frac{1}{2048\left(6 m+2 \sqrt{3}+3 \csc \left(\pi/9\right)+3 \csc \left(2 \pi/9\right)+3 \sec \left(\pi/18\right)\right)}\left[3 \csc ^3\left(\pi/18\right) \sec ^6\left(\pi/18\right) \sec ^3\left(\pi/9\right)\right.\\
    &\left(9 m\left(1+\sin \left(\pi/18\right)+2 \cos \left(2 \pi/9\right)\right)+27 \sqrt{3}+10 \sin \left(\pi/9\right)+40 \sin \left(2 \pi/9\right)+3 \sqrt{3} \sin \left(\pi/18\right)\right.\\
    &\left.\left.+77 \cos \left(\pi/18\right)+6 \sqrt{3} \cos \left(2 \pi/9\right)\right)\right],
    \end{aligned}\nonumber
\eea
\bea
\check{\eta}_+(m,9,1)=\hat{\eta}_+(m,9,1)=9 \sqrt{m(m+9)}/(6 m+2 \sqrt{3}+3 \csc \left(\pi/9\right)+3 \csc \left(2 \pi/9\right)+3 \sec \left(\pi/18\right)),\nonumber
\eea
\bea
\check{\eta}_-(m,9,1)=\hat{\eta}_-(m,9,1)=-9 \sqrt{m(m+9)}/(6 m+2 \sqrt{3}+3 \csc \left(\pi/9\right)+3 \csc \left(2 \pi/9\right)+3 \sec \left(\pi/18\right)),\nonumber
\eea
\bea
\begin{aligned}
    \check{\eta}_{\pm}(m,9,2)=&\hat{\eta}_{\pm}(m,9,2)\\
    =&\frac{1}{512 \left(6 m+2 \sqrt{3}+3 \csc \left(\pi/9\right)+3 \csc \left(2 \pi/9\right)+3 \sec \left(\pi/18\right)\right)}
    \left[27 \sin ^2\left(\pi/18\right) \csc ^9\left(\pi/9\right)\sec ^3\left(\pi/9\right)\right.\\
    &\left.\cdot\left(m\left(18 \sin \left(\pi/9\right)+10 \sqrt{3} \sin \left(\pi/18\right)-9 \cos \left(\pi/18\right)\right)-3-\sin \left(\pi/18\right)+2 \cos \left(\pi/9\right)\right)\right],
    \end{aligned}\nonumber
\eea
\bea
\begin{aligned}
    \check{\eta}_{\pm}(m,9,3)=&\hat{\eta}_{\pm}(m,9,3)\\
    =&\frac{1}{8\left(6 m+2 \sqrt{3}+3 \csc \left(\pi/9\right)+3 \csc \left(2 \pi/9\right)+3 \sec \left(\pi/18\right)\right)}
    \left[3\left(24 m-3 \csc ^3\left(2 \pi/9\right)-3 \sec ^3\left(\pi/18\right)\right.\right.\\
    &\left.\left.+8 \sqrt{3}-3\left(1+2 \cos \left(2 \pi/9\right)\right) \csc ^3\left(\pi/9\right)+6 \cos \left(\pi/9\right) \sec ^3\left(\pi/18\right)-6 \sin \left(\pi/18\right) \csc ^3\left(2 \pi/9\right)\right)\right],
    \end{aligned}\nonumber
\eea
\bea
\begin{aligned}
    \check{\eta}_{\pm}(m,9,4)=&\hat{\eta}_{\pm}(m,9,4)\\
    =&\frac{1}{2048\left(6 m+2 \sqrt{3}+3 \csc \left(\pi/9\right)+3 \csc \left(2 \pi/9\right)+3 \sec \left(\pi/18\right)\right)}\left[9 \csc ^3\left(\pi/18\right) \sec ^6\left(\pi/18\right) \sec ^3\left(\pi/9\right)\right.\\
    &\left.\cdot\left(3 m\left(1+\sin \left(\pi/18\right)+2 \cos \left(2 \pi/9\right)\right)
    -2 \left(6 \sqrt{3}+5 \sin \left(\pi/9\right)+14 \sin \left(2 \pi/9\right)+19 \cos \left(\pi/18\right)\right)\right)\right].
    \end{aligned}\nonumber
\eea
Let
\bea
\begin{aligned}
\Xi(m)=\max\left\{\zeta(\check{\alpha}_{\pm}(m,9,l),\check{\eta}_{\pm}(m,9,l)),\zeta(\hat{\alpha}_{\pm}(m,9,l),\hat{\eta}_{\pm}(m,9,l)),\ l=1,2,3,4\right\},
\end{aligned}\nonumber
\eea
with 
$$Q_{max}(9)=\csc ^3\left(\pi/18\right) \sec ^6\left(\pi/18\right) \sec ^3\left(\pi/9\right)\left(6 \sqrt{3}+5 \sin \left(\pi/9\right)+14 \sin \left(2 \pi/9\right)+19 \cos \left(\pi/18\right)\right)/2048\approx 4.9047,$$ we have
\bea
\Xi(m)=|3-(\hat{\alpha}_+(m,9,1)+\hat{\eta}_+(m,9,1))|
=\hat{\alpha}_+(m,9,1)+\hat{\eta}_+(m,9,1)-3,\nonumber
\eea
holds for $(m,e)\in (2Q_{max}(9),+\infty)\times [0,1)$.
By definition of $\mathcal{ST}_9$, we know that
\bea
\mathcal{ST}_9=\left\{(m,e)\in (0,\infty)\times [0,1)\ |\ \Xi(m)<1/\sqrt{\tilde{g}_-(e)}\right\},\nonumber
\eea
so we obtain
\bea
\begin{aligned}\label{EE1+9}
\mathcal{EE}_{(1+9)}&=\left\{(m,e)\in (2Q_{max}(9),+\infty)\times [0,1)\ |\ \Xi(m)<1/\sqrt{\tilde{g}_-(e)}\right\}\\
&=\left\{(m,e)\in (2Q_{max}(9),+\infty)\times [0,1)\ |\ \hat{\alpha}_+(m,9,1)+\hat{\eta}_+(m,9,1)-3<1/\sqrt{\tilde{g}_-(e)}
\right\}.
\end{aligned}\nonumber
\eea
The region is shown in Figure \ref{1+9.jpg}.

\hfill\newline
\noindent{\bf Acknowledgement.}
The authors thank Prof. Alain Abouy, from whose lecture notes on the two-body problem we learned about Kepler's first integrals. The authors would also like to thank Pro.Lei Zhao for providing us with the Chinese translation of Albouy's lecture notes.
This work is partially supported by the National Key R\&D Program of China (2020YFA0713303).
Y.Ou is partially supported by NSFC ($\sharp$12371192), the Young Taishan
Scholars Program of Shandong Province ($\sharp$ tsqn202312055), and the Qilu Young
Scholar Program of Shandong University.




\end{document}